\documentclass[12pt, twoside, english, headsepline]{article}

\usepackage[margin=30mm]{geometry}

\usepackage{scrlayer-scrpage}  
\clearscrheadfoot                 
\cehead[]{F. Cinque and E. Orsingher}        
\cohead[]{Stochastic Dynamics of Orthogonal Planar Random Motions}        
\cefoot[\pagemark]{\pagemark}  
\cofoot[\pagemark]{\pagemark}     

\usepackage{sectsty}
\sectionfont{\fontsize{14}{13}\selectfont}
\subsectionfont{\fontsize{12}{13}\selectfont}

\usepackage{tikz,graphicx}
\usepackage{amsmath,amsthm,amssymb,amsfonts, mathtools, amsbsy, dsfont}
\usepackage[utf8]{inputenx}
\usepackage{xcolor}
\usepackage{microtype}
\usepackage{hyperref}
\hypersetup{linkcolor=blue}

\newcommand*\dif{\mathop{}\!\mathrm{d}}  

\allowdisplaybreaks 

\newtheorem{theorem}{Theorem}[section] 
\newtheorem{corollary} {Corollary}[section] 
\newtheorem{proposition}{Proposition}[section] 

\theoremstyle{definition} 

\newtheorem{remark}{Remark}[section] 
\numberwithin{equation}{section} 

\providecommand{\keywords}[1]
{
  \small	
  \textbf{\textit{Keywords: }} #1
}
\providecommand{\MSC}[1]
{
  \small	
  \textit{2020 MSC: } #1   
}

\title{Stochastic dynamics of generalized planar random motions with orthogonal directions\footnote{\textit{Submitted to Journal of Theoretical Probability}}}
\author{Fabrizio Cinque$^1$ and Enzo Orsingher$^2$\\
        \small Department of Statistical Sciences, Sapienza University of Rome, Italy \\
        \small $^1$fabrizio.cinque@uniroma1.it $^2$enzo.orsingher@uniroma1.it
}

\begin{document}

\maketitle

\begin{abstract}
We study planar random motions with finite velocities, of norm $c>0$, along orthogonal directions and changing at the instants of occurrence of a non-homogeneous Poisson process with rate function $\lambda(t),\ t\ge0$. We focus on the distribution of the current position $\bigl(X(t), Y(t)\bigr),\ t\ge0$, in the case where the motion has orthogonal deviations and where also reflection is admitted. In all the cases the process is located within the closed square $S_{ct}=\{(x,y)\in \mathbb{R}^2\,:\,|x|+|y|\le ct\}$ and we obtain the probability law inside $S_{ct}$, on the edge $\partial S_{ct}$ and on the other possible singularities, by studying the partial differential equations governing all the distributions examined. A fundamental result is that the vector process $\bigl(X(t), Y(t)\bigr)$ is probabilistically equivalent to a linear transformation of two (independent or dependent) one-dimensional symmetric telegraph processes with rate function proportional to $\lambda(t)$ and velocity $c/2$. Finally, we extend the results to a wider class of orthogonal-type evolutions. 
\end{abstract} \hspace{10pt}


\MSC{Primary 60K99; 60G50}

\section{Introduction}

Finite speed planar random motions are a very natural class of stochastic processes to describe real movements on a two-dimensional space. We can imagine a particle randomly moving in $\mathbb{R}^2$ with $n\in \mathbb{N}$ velocities $v_k = (v_{kx}, v_{ky})\in \mathbb{R}^2,\ k=1,\dots,n$, changing at random times according to different chance mechanisms. In general, from the velocity $v_j$ it is possible to switch to an arbitrary velocity $v_k$ with probability $p_{jk}\in [0,1],\ j,k=1,\dots, n$. The changes of velocity can be cyclic, that is a deterministic law is assumed to pass from $v_k$ to $v_{k+1}$, with $v_{k+n}=v_k$, or probabilistic, when from the velocity $v_k$ one can switch to the other ones according to some random law. The switches are usually governed by a Poisson process (homogeneous as well as non-homogeneous, as in the present paper) or by some more general renewal processes, see Di Crescenzo \cite{Dc2002}.

The main goal of the investigation for planar random motions is the probability distribution of the vector process $\bigl(X(t),Y(t)\bigr)$, which describes the position, at time $t\ge0$, of the moving particle. The study of finite speed planar random processes in continuous time have been first undertaken by probabilistic and physicists such as Orsingher \cite{O1986} and Masoliver \textit{et al.} \cite{MPW1993}. In Kolesnik and Turbin \cite{KT1998} was proved the interesting connection between planar random motions with $n$ different directions and $n$-th order hyperbolic partial differential equations.

Cyclic planar motions with three directions were treated by Di Crescenzo \cite{Dc2002}, Orsingher \cite{O2002}, Leorato and Orsingher \cite{LO2004} with displacements of random length and with different kinds of mechanism governing the switches of directions.

In the last decades several papers also dealt with multidimensional evolutions, see Samoilenko \cite{S2001}. Lachal \textit{et al.} \cite{LLO2006} studied minimal cyclic random motions, i.e. motions in $\mathbb{R}^d$ with $d+1$ directions forming a regular hyperpolyhedreon, with the technique based on order statistics. The distributions obtained involve Bessel functions the prototype of which is 
$$ I_{0,n}(x) = \sum_{k=0}^\infty \Bigl(\frac{x}{n}\Bigr)^{2k}\frac{1}{k!^n},$$
with $n\in \mathbb{N}$. These results were extended by Lachal \cite{L2006}. Here the author provides a general integral formula for the distribution of a cyclic motion in $\mathbb{R}^d$ with a finite number of velocities. The study of motions in multidimensional spaces with an infinite number of directions has been carried out, for example, by Kolesnik and Orsingher \cite{KO2005}, concerning a planar evolution, and by Orsingher and De Gregorio \cite{ODg2007}, regarding higher spaces.
\\

In this paper we focus on planar random motions with orthogonal directions $d_k=v_k/|v_k|=\bigl(\cos(k\pi/2), \sin(k\pi/2)\bigr),\ |v_k|=c>0$ with $k=0,1,2,3$ (clearly $d_{k+n}=d_k$ for natural $n$) can be classified into five categories. Cyclic motion (see above), standard orthogonal motion (from direction $d_k$ the particle can switch either to $d_{k-1}$ or $d_{k+1}$), standard motion with Bernoulli trials (where the particle can also skip the change of direction), orthogonal motion with reflection (where the particle can switch either to one of the orthogonal directions or bounce back to $d_{k+2}$) and uniformly orthogonal motion (where the new direction is chosen among all the four possibile ones). The process governing the changes of direction is assumed here as a non-homogeneous Poisson process with rate $\lambda(t)$. This is a substantial difference with all the previous researches in this field. The standard orthogonal motion with $\lambda(t)=\lambda$ was examined by Orsingher and Kolesnik \cite{OK1996} and Orsingher \cite{O2000}, while the motion with reflection was considered in Kolesnik and Orsingher \cite{KO2001}. On the other hand Orsingher \textit{et al.} \cite{OGZ2020} studied the cyclic version, also in $\mathbb{R}^3$.

Some recent papers in the physical literature also treat multidimensional finite speed random motions on both the plane and higher spaces, see Mertens \textit{et al.} \cite{MetAl2012}, Elgeti and Gompper \cite{EG2015}, Hartmann \textit{et al.} \cite{HetAl2020}, Mori \textit{et al.} and Sevilla \cite{S2020}. We also recall the work of Santra \textit{et al.} \cite{SBS2020} where a section is also devoted to an orthogonal planar motion. However, explicit general laws of the distribution of the current position $\bigl(X(t),Y(t)\bigr)$ are not derived.

\begin{figure}
	\begin{minipage}{0.5\textwidth}
		\centering
		\begin{tikzpicture}[scale = 0.71]
		\draw[dashed, gray] (4,0) -- (0,4) node[above right, black, scale = 0.9]{$ct$};
		\draw[dashed, gray] (0,4) -- (-4,0) node[above left, black, scale = 0.9]{$-ct$};
		\draw[dashed, gray] (-4,0) -- (0,-4) node[below left, black, scale = 0.9]{$-ct$};
		\draw[dashed, gray] (0,-4) -- (4,0) node[above right, black, scale = 0.9]{$ct$};
		\draw[->, thick, gray] (-5,0) -- (5,0) node[below, scale = 1, black]{$\pmb{X(t)}$};
		\draw[->, thick, gray] (0,-5) -- (0,5) node[left, scale = 1, black]{ $\pmb{Y(t)}$};
		\draw (0,0)--(0.5,0)--(0.5,1)--(-0.3,1)--(-0.3,2)--(-1,2)--(-1,2.2);
		\filldraw (0,0) circle (0.8pt); \filldraw (0.5,0) circle (0.8pt); \filldraw (0.5,1) circle (1pt); \filldraw (-0.3,1) circle (0.8pt); \filldraw (-0.3,2) circle (0.8pt); \filldraw (-1,2) circle (0.8pt);
		\draw (0,0)--(0,0.4)--(-0.5,0.4)--(-0.5,1.6)--(-2.4,1.6);
		\filldraw (0,0) circle (0.8pt); \filldraw(0,0.4) circle (0.8pt); \filldraw (-0.5,0.4) circle (0.8pt); \filldraw (-0.5,1.6) circle(0.8pt);
		\draw (0,0)--(-0.8,0)--(-0.8,-2.2)--(-0.6,-2.2)--(-0.6,-2.6)--(-0.2,-2.6);
		\filldraw(-0.8,0) circle (0.8pt); \filldraw (-0.8,-2.2) circle (0.8pt); \filldraw(-0.6,-2.2) circle(0.8pt); \filldraw(-0.6,-2.6) circle (0.8pt);
		\draw (0,0)--(0,-0.8)--(1,-0.8)--(1,-0.9)--(1.4,-0.9)--(1.8,-0.9)--(1.8, 0.2);
		\filldraw(0,-0.8)circle (0.8pt); \filldraw(1,-0.8)circle (0.8pt); \filldraw(1,-0.9)circle (0.8pt); \filldraw(1.4,-0.9)circle (0.8pt); \filldraw(1.8,-0.9)circle (0.8pt); 
		\end{tikzpicture}
		\caption{\small Sample paths of a standard \newline orthogonal planar motion.}\label{MPS_1}
	\end{minipage}\hfill
	\begin{minipage}{0.5\textwidth}
		\centering
		\begin{tikzpicture}[scale = 0.71]
		\draw[dashed, gray] (4,0) -- (0,4) node[above right, black, scale = 0.9]{$ct$};
		\draw[dashed, gray] (0,4) -- (-4,0) node[above left, black, scale = 0.9]{$-ct$};
		\draw[dashed, gray] (-4,0) -- (0,-4) node[below left, black, scale = 0.9]{$-ct$};
		\draw[dashed, gray] (0,-4) -- (4,0) node[above right, black, scale = 0.9]{$ct$};
		\draw[->, thick, gray] (-5,0) -- (5,0) node[below, scale = 1, black]{$\pmb{X(t)}$};
		\draw[->, thick, gray] (0,-5) -- (0,5) node[left, scale = 1, black]{ $\pmb{Y(t)}$};
		\draw (0,0)--(0.2,0)--(0.2,1.5)--(0.2,1.2)--(1.2,1.2)--(1.2,0.95)--(1.2,1.05)--(1.6,1.05);
		\filldraw (0,0)circle (0.8pt); \filldraw(0.2,0)circle (0.8pt); \filldraw(0.2,1.5)circle (0.8pt); \filldraw(0.2,1.2)circle (0.8pt); \filldraw(1.2,1.2)circle (0.8pt); \filldraw(1.2,0.95)circle (0.8pt); \filldraw(1.2,1.05)circle (0.8pt); 
		\draw (0,0)--(0,0.4)--(-0.2,0.4)--(-0.2, 0.55)--(-0.2,0.25)--(-0.2,0.55)--(-1,0.55)--(-1, 0.4)--(-2.7,0.4);
		\filldraw(0,0.4)circle (0.8pt); \filldraw(-0.2,0.4)circle (0.8pt); \filldraw(-0.2, 0.55);\filldraw(-0.2,0.25)circle (0.8pt); \filldraw(-0.2,0.55)circle (0.8pt); \filldraw(-1,0.55)circle (0.8pt); \filldraw(-1, 0.4)circle (0.8pt); 
		\draw (0,0)--(-0.8,0)--(-0.8,-2)--(-0.8,-1.3)--(-1.1,-1.3);
		\filldraw(-0.8,0)circle (0.8pt); \filldraw(-0.8,-2)circle (0.8pt); \filldraw(-0.8,-1.3)circle (0.8pt); 
		\draw (0,0)--(0,-0.8)--(1,-0.8)--(1,-1)--(1.7,-1)--(1.4,-1)--(1.4, -1.4)--(1.4,-1.2)--(2,-1.2);
		\filldraw(0,-0.8)circle (0.8pt); \filldraw(1,-0.8)circle (0.8pt); \filldraw(1,-1)circle (0.8pt); \filldraw(1.7,-1)circle (0.8pt); \filldraw(1.4,-1)circle (0.8pt); \filldraw(1.4, -1.4)circle (0.8pt); \filldraw(1.4,-1.2)circle (0.8pt); 
		\end{tikzpicture}
		\caption{\small Sample paths of a reflecting orthogonal planar motion.}\label{MPS_2}
	\end{minipage}
\end{figure}

In the orthogonal case we are able to obtain general laws for the vector process $\{\bigl(X(t),Y(t)\bigr)\}_{t\ge0}$ including the case where the switches of directions are governed by a non-homogeneous Poisson process. In this case we are able to prove that the current position $\bigl(X(t),Y(t)\bigr)$, for the standard case, can be represented as 
\begin{equation}\label{decomposizioneIntro}
\begin{cases}
X(t) = U(t) + V(t)\\
Y(t) = U(t) - V(t)
\end{cases}
\end{equation}
where $U=\{U(t)\}_{t\ge0}$ and $V=\{V(t)\}_{t\ge0}$ are independent one-dimensional telegraph process. Their absolutely continuous components $p=p_U(x,t) = p_V(x,t)$ are solution to
\begin{equation}\label{equazioneTelegrafoIntroduzione}
\frac{\partial^2p}{\partial t^2} +\lambda(t)\frac{\partial p}{\partial t} = \frac{c^2}{4} \frac{\partial^2 p}{\partial x^2}.
\end{equation}
This means that $U$ and $V$ are (independent) telegraph processes with rate function $\lambda(t)/2$ and velocity $c/2$. This extends to the non-homogeneous case a previous result for the homogeneous Poisson process governing the switches of direction, see Orsingher \cite{O2000}. For all cases where the telegraph equation (\ref{equazioneTelegrafoIntroduzione}) can be treated, it is possible to arrive at the general law
\begin{equation}
P\{X(t)\in \dif x, \, Y(t)\in \dif y\}/(\dif x\dif y) = p(x,y,t) = \frac{1}{2}p_U\Bigl(\frac{x+y}{2}\Bigr)p_V\Bigl(\frac{x-y}{2}\Bigr),
\end{equation}
for $(x,y)\in S_{ct} = \{(x,y)\in\mathbb{R}^2\,:\,|x|+|y|\le ct\}$. 

If $\Lambda(t) = \int_0^t\lambda(s)\dif s<\infty,\ t\ge0,$ the particle can reach the edge $\partial S_{ct}$ of its support and in this case we have also a direct derivation of the probability law, for example, on the side of $\partial S_{ct}$ belonging to the first quadrant, $\partial S_{ct}\cap (0,\infty)\times(0,\infty)$,
$$f(\eta, t)\dif \eta = P\{X(t)+Y(t)=ct,\, X(t)-Y(t)\in \dif \eta\},\ \ \ |\eta|<ct,$$
This probability satisfies the following second-order differential system 
\begin{equation}\label{sistemaFrontieraAltoDx}
\begin{cases}
\frac{\partial^2 f}{\partial t^2} +2\lambda(t)\frac{\partial f}{\partial t} +\frac{1}{2}\Bigl(\frac{3}{2}\lambda(t)^2+\lambda'(t)\Bigr)f =c^2\frac{\partial^2 f}{\partial \eta^2},\ \ \ |\eta|<ct,\\
f(\eta,t)\ge0,\\
\int_{-ct}^{ct}f(\eta,t)\dif \eta = \frac{1}{2}(e^{-\Lambda(t)/2}-e^{-\Lambda(t)}).
\end{cases}
\end{equation} We note that
$$P\{\bigl(X(t),Y(t)\bigr)\in \partial S_{ct}\} = 2\Bigl(e^{-\frac{1}{2}\Lambda(t)}-e^{-\Lambda(t)}\Bigr)+e^{-\Lambda(t)},$$
where the first component pertains the sides of $\partial S_{ct}$ and the last term is  the probability of reaching the vertexes of $S_{ct}$.

Another important probabilistic information concerns the distribution on squares of half-diagonal $0\le z\le ct$, which reads
\[P\{|X(t)|+|Y(t)|\in \dif z\} = 4 p_U\Bigl(\frac{z}{2}, t \Bigr)\int_0^{\frac{z}{2}} p_V(w,t)\dif w,\]
this can be interpreted as the distribution of the $L_1$-distance of the standard orthogonal process.

A complete picture of the motion is achieved thanks to the analysis of the marginal components. We give the third-order partial differential equation governing the projections $X(t)$ and $Y(t)$ and the characteristics of the one-dimensional motion they describe. The distribution $P\{X(t)\in \dif x\}$, with $|x|<ct$, is sometimes a hard technical problem which we tackle in some special cases.

Finally, we are able to extend the results of the standard orthogonal motion to the motion with Bernoulli trials and with different velocities along the two axes, thus producing some asymmetry.

The third section of the paper concerns the planar motion with reflection. This is substantially different from the standard one because at each switch of direction the particle can either deviate orthogonally or bounce back. This makes the distribution of the position process more complicated because of the appearance of an additional singularity along the diagonals of the support $S_{ct}$. One of the main consequences of the possible reflection of the particle is that a decomposition of the form (\ref{decomposizioneIntro}) makes the processes $U$ and $V$ dependent. Also here these are one-dimensional telegraph processes and we are able to describes their relationship.

When the rate function is constant, $\lambda(t)=\lambda>0\ \forall \ t\ge0$, Kolesnik and Orsingher \cite{KO2001} obtained the distribution on both the edge and the diagonals of $S_{ct}$. Here we extend these results to the case of a non-homogeneous Poisson process governing the switches. In particular, we are again able to connect the probabilities of the planar motion to those of one-dimensional telegraph processes.

At last, we extend the results of the reflecting planar motion to the uniformly orthogonal motion and to a wider class of orthogonal processes.

\section{Standard orthogonal planar random motion}

In this section we consider planar motions with directions $d_k = v_k/|v_k| = (\cos(k\pi/2), \sin(k\pi/2)), \\ k=0,1,2,3$ where $|v_k|=c>0$. The Poisson process governing the changes of direction is non-homogeneous with rate function $\lambda(t)\ge0, \ t\ge0$. At each Poisson event the moving particle can switch to one of the orthogonal directions with probability $1/2$.
\\We denote by $\{\bigl(X(t),Y(t)\bigr)\}_{t\ge0},$ the current position of the moving particle and $\{N(t\}_{t\ge0}$ the number of changes of direction recorded up to time $t$. 

\subsection{The governing partial differential equation}\label{sottoSezioneEquazioneStandard}

\begin{theorem}\label{teoremaEquazioneStandard}
The absolutely continuous component $p = p(x,y,t)\in C^4\bigl(\mathbb{R}^2\times[0,\infty), [0,\infty)\bigr)$ of the distribution of the standard orthogonal process $\{\bigl(X(t),Y(t)\bigr)\}_{t\ge0}$ satisfies the following fourth-order differential equation with time-varying coefficients
\begin{equation}\label{equazioneStandard}
\biggl(\frac{\partial^2 }{\partial t^2}+2\lambda\frac{\partial}{\partial t} +\lambda^2+\lambda'\biggr)\biggl(\frac{\partial^2}{\partial t^2} +2\lambda\frac{\partial}{\partial t} -c^2\Bigl(\frac{\partial^2}{\partial x^2}+\frac{\partial^2}{\partial y^2}\Bigr)\biggr)p +c^4\frac{\partial^4 p}{\partial x^2\partial y^2} = \lambda'\frac{\partial^2 p}{\partial t^2}+(\lambda''+\lambda\lambda')\frac{\partial p}{\partial t}
\end{equation}
where $\lambda=\lambda(t)\in C^2\bigl((0,\infty),[0,\infty)\bigl)$ denotes the rate function of the non-homogeneous Poisson process governing the changes of direction.
\end{theorem}

\begin{proof}
We use the following notations
\begin{equation}\label{notazione}
f_k(x,y,t) \dif x\dif y = P\{X(t)\in \dif x, Y(t)\in \dif y,D(t) = d_k\},\ \ \ k=0,1,2,3,
\end{equation}
and $\{D(t)\}_{t>0}$ is the process taking the four directions $d_k$. We remark that at Poisson times the particle moving with direction $d_k$ can pass to $d_{k+1},d_{k-1}$ ($d_{k+4} = d_k = d_{k-4}$) with equal probability. We observe that, for $t\ge0$ and $(x,y)\in S_{c(t+\dif t)}$
\[f_0(x,y,t+\dif t) = f_0(x-c\dif t, t)\bigl(1-\lambda(t)\dif t \bigr) + \bigl[f_1(x,y-c\dif t,t )+f_3(x,y+c\dif t,t) \bigr]\lambda(t)\dif t+ o(\dif t) \]
and similarly for $f_1,f_2$ and $f_3$. With this at hand, we obtain the differential system governing the probability densities (\ref{notazione}),
\begin{equation}\label{sistemaStandard}
\begin{cases}
\frac{\partial f_0}{\partial t} = -c\frac{\partial f_0}{\partial x}+\frac{\lambda(t)}{2}(f_1+f_3-2f_0),\\
\frac{\partial f_1}{\partial t} = -c\frac{\partial f_1}{\partial y}+\frac{\lambda(t)}{2}(f_2+f_0-2f_1),\\
\frac{\partial f_2}{\partial t} = c\frac{\partial f_2}{\partial x}+\frac{\lambda(t)}{2}(f_1+f_3-2f_2),\\
\frac{\partial f_3}{\partial t} = c\frac{\partial f_3}{\partial y}+\frac{\lambda(t)}{2}(f_2+f_0-2f_3),
\end{cases} \text{and then\ \ \ } 
\begin{cases}
\frac{\partial g_1}{\partial t} = -c\frac{\partial g_2}{\partial x}+\lambda(t) (g_3-g_1),\\
\frac{\partial g_2}{\partial t} = -c\frac{\partial g_1}{\partial x}-\lambda(t) g_2,\\
\frac{\partial g_3}{\partial t} = -c\frac{\partial g_4}{\partial y}+\lambda(t) (g_1-g_3),\\
\frac{\partial g_4}{\partial t} = -c\frac{\partial g_3}{\partial y}-\lambda(t) g_4.
\end{cases}
\end{equation}
where we used the following transformation to obtain the second system appearing in (\ref{sistemaStandard})
$$f_0+f_2 =g_1, \ \ \ f_0-f_2 = g_2, \ \ \ f_1+ f_3=g_3, \ \ \ f_1-f_3 = g_4.$$
Put $\lambda = \lambda(t)$. By differentions and substitutions we pass from the system of four first-order equations (\ref{sistemaStandard}) to the following system of second-order equations with two unknown functions $g_1$ and $g_3$
\begin{equation}\label{sistemaSecondoOrdineStandard}
\begin{cases}
\frac{\partial^2 g_1}{\partial t^2} = c^2 \frac{\partial^2 g_1}{\partial x^2} -2\lambda\frac{\partial g_1}{\partial t} +\lambda\frac{\partial g_3}{\partial t}+ (\lambda^2 +\lambda')(g_3-g_1),\\
\frac{\partial^2 g_3}{\partial t^2} = c^2 \frac{\partial^2 g_3}{\partial y^2} -2\lambda\frac{\partial g_3}{\partial t} +\lambda\frac{\partial g_1}{\partial t}+ (\lambda^2 +\lambda')(g_1-g_3).
\end{cases}
\end{equation}
By summing up and subtracting equations (\ref{sistemaSecondoOrdineStandard}) we obtain
\begin{equation}\label{sistemapwStandard}
\begin{cases}
\frac{\partial^2 p}{\partial t^2} = \frac{c^2}{2} \biggl[\Bigl(\frac{\partial^2}{\partial x^2}+\frac{\partial^2}{\partial y^2}\Bigr)p+\Bigl(\frac{\partial^2}{\partial x^2}-\frac{\partial^2}{\partial y^2}\Bigr)w\biggr]-\lambda\frac{\partial p}{\partial t},\\[8pt]
\frac{\partial^2 w}{\partial t^2} = \frac{c^2}{2} \biggl[\Bigl(\frac{\partial^2}{\partial x^2}-\frac{\partial^2}{\partial y^2}\Bigr)p+\Bigl(\frac{\partial^2}{\partial x^2}+\frac{\partial^2}{\partial y^2}\Bigr)w\biggr]-3\lambda\frac{\partial w}{\partial t}-2(\lambda^2+\lambda')w,
\end{cases}
\end{equation}
where $p=g_1+g_3$ and $w = g_1-g_3$. We outline the method to pass from (\ref{sistemapwStandard}) to (\ref{equazioneStandard}). By means of the following differential operators
\begin{equation*}
\begin{array}{l}
\Delta^+ = \frac{c^2}{2}\Bigl(\frac{\partial^2}{\partial x^2}+\frac{\partial^2}{\partial y^2}\Bigr),\ \ \ F = \Delta^+ -\lambda\frac{\partial }{\partial t},
\\
\Delta^- = \frac{c^2}{2}\Bigl(\frac{\partial^2}{\partial x^2}-\frac{\partial^2}{\partial y^2}\Bigr), \ \ \ G =\Delta^+  -3\lambda\frac{\partial }{\partial t}-2(\lambda^2+\lambda'),
\end{array}
\end{equation*}
the system (\ref{sistemapwStandard}) reads
\begin{equation}\label{sistemapwStandardSemplificato}
\begin{cases}
\frac{\partial^2 p}{\partial t^2} = Fp+\Delta^-w,\\
\frac{\partial^2 w}{\partial t^2} =\Delta^-p +Gw.
\end{cases}
\end{equation}
The elimination of $w$ from (\ref{sistemapwStandardSemplificato}) is performed by taking the second-order time derivative and by considering the commutativity of the differential operators
\begin{align}\label{}
\frac{\partial^4p}{\partial t^4}&=\frac{\partial^2 }{\partial t^2}Fp+\Delta^-\frac{\partial^2 }{\partial t^2}w = \frac{\partial^2 }{\partial t^2}Fp+\Delta^-\bigl(\Delta^-p +Gw\bigr) \nonumber\\
&= \frac{\partial^2 }{\partial t^2}Fp + (\Delta^-)^2p +G(\Delta^-w )=  \frac{\partial^2 }{\partial t^2}Fp + (\Delta^-)^2p +G\Bigl(\frac{\partial^2 p}{\partial t^2}-Fp\Bigr).\nonumber
\end{align}
After some calculation, we obtain 
\begin{align}\label{equazioneStandardEsplicita}
\frac{\partial^4p}{\partial t^4} + 4\lambda \frac{\partial^3p}{\partial t^3}+(5\lambda+4\lambda')\frac{\partial^2p}{\partial t^2}+ (2\lambda^2+5\lambda\lambda'+\lambda'')\frac{\partial p}{\partial t} + c^4\frac{\partial^4 p}{\partial x^2\partial y^2}\\
=c^2\Bigl(\frac{\partial^2}{\partial x^2}+\frac{\partial^2}{\partial y^2}\Bigr)&\Bigl(\frac{\partial^2p}{\partial t^2}+2\lambda\frac{\partial p}{\partial t} + (\lambda^2+\lambda')p\Bigr)\nonumber
\end{align}
and this coincides with the claimed result (\ref{equazioneStandard}).
\end{proof}

If $\lambda(t) = \lambda>0\ \forall\ t$, equation (\ref{equazioneStandard}) reduces to
\begin{equation}\label{equazioneStandardCasoCostante}
\Bigl(\frac{\partial^2 }{\partial t^2}+2\lambda\frac{\partial}{\partial t} +\lambda^2\Bigr)\biggl(\frac{\partial^2}{\partial t^2} +2\lambda\frac{\partial}{\partial t} -c^2\Bigl(\frac{\partial^2}{\partial x^2}+\frac{\partial^2}{\partial y^2}\Bigr)\biggr)p +c^4\frac{\partial^4 p}{\partial x^2\partial y^2} = 0
\end{equation}
which coincides with (3.6) of Orsingher and Kolesnik \cite{OK1996}..

It is well-known that under Kac's conditions, i.e. if $\lambda,c\longrightarrow\infty$ such that $\lambda/c^2\longrightarrow \sigma^2$, the standard orthogonal planar motion converges to a planar Brownian motion with diffusivity $\sigma^2$ (it is sufficient to divide by $\lambda^3$ equation (\ref{equazioneStandardCasoCostante})). 

\begin{remark}
Let $\lambda=\lambda(t)\in C^3\bigl((0,\infty),[0,\infty)\bigl)$. By means of the transformation $p(x,y,t) = e^{-\int_0^t\lambda(s)\dif s}u(x,y,t)$ with $\int_0^t\lambda(s)\dif s<\infty$, we obtain the fourth order partial differential equation in $u$
\begin{equation*}\label{}
\biggl(\frac{\partial^2}{\partial t^2} -\lambda^2 -c^2\Bigl(\frac{\partial^2}{\partial x^2}+\frac{\partial^2}{\partial y^2}\Bigr)\biggr)\frac{\partial^2u}{\partial t^2}+c^4\frac{\partial^4 u}{\partial x^2\partial y^2} = 2\lambda'\frac{\partial^2u}{\partial t^2}+3(\lambda''+\lambda')\frac{\partial u}{\partial t}+ (\lambda'^2+\lambda\lambda'' +\lambda''')u.
\end{equation*}
\hfill$\diamond$
\end{remark}

\begin{remark}[Standard motion with Bernoulli trials]\label{remarkStandardBernoulliEquazione}
Let us consider a planar orthogonal random motion that behaves as the standard one with the upgrade that, at the Poisson times, it can continue to move along the same direction with probability $1-q\in[0,1)$. This means that at each occurrence of the Poisson events, the particle deviates on each orthogonal direction with probability $q/2$. This extension permits us to consider a refracting behavior for the particle. We call it \textit{$q$-standard motion}, $q\in(0,1]$ (or Standard motion with Bernoulli trials).

By proceeding as above in the case of the standard orthogonal motion, it can be proved that the absolutely continuous component of this generalized motion satisfies equation (\ref{equazioneStandard}) (or equivalently (\ref{equazioneStandardEsplicita})) with rate function $\lambda_q(t) = q\lambda(t)$ instead of $\lambda(t)$. This can be proved by calculating the differential system corresponding to (\ref{sistemaStandard}). Here the rate function $q\lambda(t)$ replaces $\lambda(t)$.

Under Kac's conditions, the $q$-standard motion converges to planar Brownian motion with diffusivity $\sigma^2/q$.\hfill$\diamond$
\end{remark}

\begin{remark}[Asymmetric motion]\label{remarkEquazioneAsimmetrico}
If we assume that the particle can move on the $x$-axis with velocity $c_X=|v_1|=|v_3|>0$ and on the $y$-axis with velocity $c_Y = |v_2|=|v_3|>0$ slight changes are needed. We can ``symmetrize'' the velocities by considering the new space coordinates $(x',y')$ which are related to the original ones by means of the scaling 
$$x' = x/c_X,\ \ \  y' = y/c_Y.$$
These assumptions permit us to consider an asymmetric behavior of the particle. The support of the asymmetric motion $\bigl(X(t),Y(t)\bigr)$ is the rhombus
\begin{equation*}
 R_{t} = \Big\{(x,y)\in \mathbb{R}^2\,:\,\frac{|x|}{c_X}+\frac{|y|}{c_Y}\le t\Big\}.
\end{equation*}
The fourth-order differential equation governing the density of the absolutely continuous component of the asymmetric motion is obtained by performing the following substitutions in (\ref{equazioneStandard})
\[ c^2\Bigl(\frac{\partial^2}{\partial x^2}+\frac{\partial^2}{\partial y^2}\Bigr) \rightarrow c^2_X\frac{\partial^2}{\partial x^2}+c^2_Y\frac{\partial^2}{\partial y^2}, \ \ \ \ c^4\frac{\partial^4 u}{\partial x^2\partial y^2}\rightarrow c_X^2c_Y^2\frac{\partial^4 u}{\partial x^2\partial y^2}. \]
\\Note that all these considerations can be applied to the standard motion with Bernoulli trials as well.
\hfill$\diamond$
\end{remark}

\subsection{Explicit representation of the probability distribution}

As we previously observed, the particle performing the standard planar random motion $\{\bigl(X(t),Y(t)\bigr)\}_{t\ge0}$ is located at time $t>0$ in the square
\begin{equation}\label{supportoQuadrato}
 S_{ct} = \{(x,y)\in \mathbb{R}^2\,:\,|x|+|y|\le ct\}.
\end{equation}
Furthermore, at time $t$ it lies on the border $\partial S_{ct}$ with probability 
\begin{align}\label{massaFrontieraCasoStandard}
P\big\{\bigl(X(t),Y(t)\bigr)\in \partial S_{ct}\big\} &= P\{N(t) = 0\}+\sum_{k=1}^\infty \frac{1}{2^{k-1}}P\{N(t) = k\} \\
&=2e^{-\frac{1}{2}\int_0^t\lambda(s)\dif s} -e^{-\int_0^t\lambda(s)\dif s},\nonumber
\end{align}
provided that $\int_0^t\lambda(s)\dif s<\infty$. Note that the probability of being on the vertices, $V_{ct} = \{(0,\pm ct),(\pm ct,0)\}$, is equal to $P\{N(t)=0\}$.
\\This means that if $\int_0^t\lambda(s)\dif s=\infty$ the moving particle will be inside the square $S_{ct}$ with probability one and $\bigl(X(t),Y(t)\bigr)$ is an absolutely continuous random vector for all $t>0$.

On each side of  $\partial S_{ct}$ the moving particle performs a telegraph process (see below for further explanations).

We note that the absolutely continuous part of the distribution inside $S_{ct}$ and on the edge $\partial S_{ct}$ coincide at the instant $t^*>0$ such that
\[\int_0^{t^*}\lambda(s)\dif s =  -2 \ln\Bigl(1-\frac{1}{\sqrt{2}}\Bigr).\]

Finally, we observe that  in order to have only a singular component on the front edge $\partial S_{ct}$ we must assume that the particle can choose with probability $1/4$ one of the couples of direction $(d_k, d_{k+1})$, with $k=0,1,2,3$ and $d_4 = d_0$, and then move alternatively with the directions of the chosen couple and with switches occurring at Poisson times.

In the following proposition we study the distribution of the motion on the border $\partial S_{ct}$. In this case we deal with a particle that is continuously forced to reach a set which is itself moving outwards. This is the main source of difficulty of the next statement.

\begin{proposition}\label{propFrontieraStandard}
Let $\{\bigl(X(t),Y(t)\bigr)\}_{t\ge0}$ be a standard orthogonal planar random motion with changes of direction paced by a Poisson process with rate $\lambda(t)\in C^1\bigl((0,\infty),[0,\infty)\bigl)$ such that $\Lambda(t) = \int_0^t\lambda(s)\dif s<\infty$ for all $t>0$. Then, for $|\eta|<ct$
\begin{equation}\label{leggeFrontieraAltoDx}
f(\eta, t)\dif \eta =P\{X(t)+Y(t) = ct, X(t)-Y(t)\in \dif \eta\}
\end{equation}
satisfies the differential problem
\begin{equation}\label{sistemaFrontieraAltoDx}
\begin{cases}
\frac{\partial^2 f}{\partial t^2} +2\lambda(t)\frac{\partial f}{\partial t} +\frac{1}{2}\Bigl(\frac{3}{2}\lambda(t)^2+\lambda'(t)\Bigr)f =c^2\frac{\partial^2 f}{\partial \eta^2},\\
f(\eta,t)\ge0,\\
\int_{-ct}^{ct}f(\eta,t)\dif \eta = \frac{1}{2}(e^{-\Lambda(t)/2}-e^{-\Lambda(t)}).
\end{cases}
\end{equation}
\end{proposition}
Note that (\ref{leggeFrontieraAltoDx}) is the probability that the motion lies on the border of its support in the first quadrant, that is
\begin{equation}\label{insiemeFrentieraAltoDx}
\partial F_{ct} = \partial S_{ct}^{(1)}\setminus V_{ct} = \{(x,y)\in \mathbb{R}^2\,:\, x+y = ct, |x-y|<ct\}
\end{equation}
where $V_{ct} = \{(ct,0),(0,ct),(-ct,0),(0,-ct)\}$.

Clearly, equivalent results hold for the other components of the border $\partial S_{ct}$.

\begin{proof}
In order that the particle reaches at time $t>0$ a point on the set $\partial F_{ct}$ it must alternate between rightward ($d_0$) and upward ($d_1$) displacements for the whole time interval $[0,t]$.
\\In order to describe the probabilistic behavior on this border, we need
\begin{equation}\label{notazioneFrontieraStandard}
\begin{cases}
f_0(\eta, t) \dif \eta = P\{X(t)+Y(t) = ct,\ X(t)-Y(t)\in \dif \eta,\ D(t) = d_0\},\\
f_1(\eta, t) \dif \eta = P\{X(t)+Y(t) = ct,\ X(t)-Y(t)\in \dif \eta,\ D(t) = d_1\}.
\end{cases}
\end{equation}
We note that if the particle at time $t$ is moving rightward and was on $\partial F_{ct}$, that is $x+y=ct$, then at time $t+\dif t$ it will be on $\partial F_{c(t+\dif t)}$ with coordinates $(x+c\dif t, y)$ that is $X(t+\dif t)-Y(t+\dif t)  =x+c\dif t-y = \eta+\dif\eta\ (\dif \eta = c\dif t)$ and similarly for the upward movements. As a consequence we have that
\[f_0(\eta, t+\dif t) = f_0(\eta-c\dif t,t )(1-\lambda(t)\dif t)+ f_1(\eta +c\dif t,t)\frac{1}{2}\lambda(t)\dif t.\]
By means of simple calculation we obtain
\begin{equation}
\begin{cases}
\frac{\partial f_0}{\partial t} = -c\frac{\partial f_0}{\partial \eta}+\frac{\lambda(t)}{2}(f_1-2f_0),\\
\frac{\partial f_1}{\partial t} = c\frac{\partial f_1}{\partial \eta}+\frac{\lambda(t)}{2}(f_0-2f_1),
\end{cases}\text{and}\ \ \
\begin{cases}
\frac{\partial f}{\partial t} = -c\frac{\partial w}{\partial \eta}-\frac{\lambda(t)}{2}f,\\
\frac{\partial w}{\partial t} = -c\frac{\partial f}{\partial \eta}-\frac{3}{2}\lambda(t) w,
\end{cases}
\end{equation}
where $f=f_0+f_1$ (it coincides with (\ref{leggeFrontieraAltoDx})) and $w = f_0-f_1$. From the above system we extract the second-order  differential equation appearing in (\ref{sistemaFrontieraAltoDx}).

The integral condition of (\ref{sistemaFrontieraAltoDx}) follows from the fact that
\begin{align*}
\int_{-ct}^{ct}f(\eta,t)\dif \eta = P\big\{\bigl(X(t),Y(t)\bigr)\in \partial F_{ct} \big\} &= \frac{1}{4}P\big\{\bigl(X(t),Y(t)\bigr)\in \partial S_{ct}\setminus V_{ct} \big\}\\
&=\frac{1}{2}(e^{-\Lambda(t)/2}-e^{-\Lambda(t)}) 
\end{align*}
which can be easily derived from formula (\ref{massaFrontieraCasoStandard}).
\end{proof}

\begin{remark}
In the special case where $\lambda(t)=\lambda\ \forall \ t$, the system (\ref{sistemaFrontieraAltoDx}) reduces to
\begin{equation}\label{sistemaFrontieraAltoDxCostante}
\begin{cases}
\frac{\partial^2 f}{\partial t^2} +2\lambda\frac{\partial f}{\partial t} +\frac{3}{4}\lambda^2f =\frac{c^2}{4}\frac{\partial^2f}{\partial \eta^2},\\
f(\eta,t)\ge 0,\\
\int_{-ct}^{ct}f(\eta,t)\dif \eta = \frac{1}{2}(e^{-\frac{\lambda t}{2}}-e^{-\lambda t}).
\end{cases}
\end{equation}
The differential equation above can be further reduced to the Klein-Gordon equation by means of the transformation $f(\eta, t) = e^{-\lambda t}q(\eta, t)$,
\[\frac{\partial^2 q}{\partial t^2} -\frac{c^2}{4}\frac{\partial^2 q}{\partial \eta^2}=\frac{\lambda^2}{4}q,\ \ \ \ |\eta|<ct.\]
Hence, the solution of (\ref{sistemaFrontieraAltoDxCostante}) is
\[f(\eta, t)=\frac{e^{-\lambda t}}{4c}\Biggl[\frac{\lambda}{2} I_0\Bigl(\frac{\lambda}{2c}\sqrt{c^2t^2-\eta^2} \Bigr)+\frac{\partial}{\partial t}I_0\Bigl(\frac{\lambda}{2c}\sqrt{c^2t^2-\eta^2} \Bigr)\Biggr]. \]
From this result, we can also extract the conditional distributions with respect to the number of switches, $N(t)$. For integer $k\ge0$,
\begin{align}
 P\{X(t)+Y(t)& = ct, X(t)-Y(t)\in \dif \eta\,|\,N(t)=2k+1\} \nonumber\\
&=2P\{X(t)+Y(t) = ct, X(t)-Y(t)\in \dif \eta\,|\,N(t)=2k+2\}\nonumber\\
&=\frac{(2k+1)!}{2\,k!^2}\frac{(c^2t^2-\eta^2)^k}{(4ct)^{2k+1}}\dif \eta \nonumber
\end{align}
where $|\eta|<ct$.\hfill$\diamond$
\end{remark}

Our main result is stated in the following theorem.
\begin{theorem}\label{teoremaDecomposizioneXYUV}
The standard planar orthogonal random motion $\{\bigl(X(t),Y(t)\bigr)\}_{t\ge0}$, with rate function $\lambda(t)\in C^2\bigl((0,\infty),[0,\infty)\bigl)$, is equal in distribution to the linear transformation of two independent one-dimensional telegraph processes $\{U(t)\}_{t\ge0},\ \{V(t)\}_{t\ge0}$ with parameters $(c/2, \lambda(t)/2)$
\begin{equation}\label{decomposizioneXYUV}
\begin{cases}
X (t)= U(t)+V(t),\\
Y(t) = U(t)-V(t).
\end{cases}
\end{equation}
\end{theorem}
The statement (\ref{decomposizioneXYUV}) shows that $\bigl(X(t),Y(t)\bigr)$ is a rotation of $45^\text{o}$ of the vector with independent components $\bigl(U(t),V(t)\bigr)$.

Note that the absolutely continuous component $p(x,y,t)\dif x\dif y \\= P\{X(t)\in \dif x,Y(t)\in \dif y\}$ is given by 
\begin{equation}\label{}
p(x,y,t)=\frac{1}{2}p_U\Bigl(\frac{x+y}{2},t\Bigr)p_V\Bigl(\frac{x-y}{2},t\Bigr)
\end{equation}
where $p_U$ and $p_V$ are the densities of one-dimensional telegraph processes, i.e. they are solutions of the Cauchy problem
\begin{equation}\label{CauchyTelegrafo}
\begin{cases}
\frac{\partial^2 u}{\partial t^2} +\lambda(t)\frac{\partial u}{\partial t} =\frac{c^2}{4}\frac{\partial^2 u}{\partial x^2}, \\
u(x,0) = \delta(x),\ \ \ \frac{\partial u}{\partial t}(x,t)\big|_{t=0} = 0,
\end{cases}
\end{equation}
where $\delta$ is the Dirac delta function centered in $x=0$.

\begin{proof}
If $\Lambda(t) =\int_0^t\lambda(s)\dif s<\infty$ the distribution of $\bigl(X(t),Y(t)\bigr),\ t\ge0,$ has a positive probability mass on $\partial S_{ct}$. It is trivial to show that, for $|\eta|<ct$, 
\[f(\eta,t) = \frac{e^{-\frac{\Lambda(t)}{2}}}{2}\,p_V\Bigl(\frac{\eta}{2}, t\Bigr)\]
 satisfies (\ref{sistemaFrontieraAltoDx}). The probability on the vertices of the border $\partial S_{ct}$ easily follows as well.

For the absolutely continuous component of the distribution, it is sufficient to prove that $p = p(x,y,t)$  satisfies (\ref{equazioneStandard}). In order to spare calculation we pass to variables
\begin{equation}\label{cambioVariabiliLeggeEsplicita}
u = \frac{x+y}{2},\ \ \ v = \frac{x-y}{2} 
\end{equation}
and prove that 
\begin{equation}\label{p_uv}
q(u,v,t)=\frac{1}{2}p_U(u,t)p_V(v,t)=p(u+v,u-v,t)
\end{equation}
(disregarding the factor $1/2$) satisfies, with $\lambda=\lambda(t)$,
\begin{equation}\label{equazioneStandard_uv}
\biggl(\frac{\partial^2 }{\partial t^2}+2\lambda\frac{\partial}{\partial t} +\lambda^2+\lambda'\biggr)\biggl(\frac{\partial^2}{\partial t^2} +2\lambda\frac{\partial}{\partial t} -\frac{c^2}{2}\Bigl(\frac{\partial^2}{\partial u^2}+\frac{\partial^2}{\partial v^2}\Bigr)\biggr)q +\frac{c^4}{16}\Bigl(\frac{\partial^2}{\partial u^2}-\frac{\partial^2}{\partial v^2}\Bigr)^2q = \lambda'\frac{\partial^2 q}{\partial t^2}+(\lambda''+\lambda\lambda')\frac{\partial q}{\partial t}
\end{equation}
where $p_U$ and $p_V$ satisfy (\ref{CauchyTelegrafo}). Equation (\ref{equazioneStandard_uv}) is obtained from (\ref{equazioneStandard}) by means of the change of variables (\ref{cambioVariabiliLeggeEsplicita}).
\\From (\ref{p_uv}) we have that
\begin{equation}\label{derivata2quv0}
\frac{\partial^2 q}{\partial t^2}=\frac{c^2}{4}\Bigl(\frac{\partial^2}{\partial u^2}+\frac{\partial^2}{\partial v^2}\Bigr)q - \lambda\frac{\partial q}{\partial t}+2\frac{\partial p_U}{\partial t}\frac{\partial p_V}{\partial t}
\end{equation}
and then, by multiplying by $2$ both members of (\ref{derivata2quv0}) we obtain
\begin{equation}\label{derivata2quv}
\frac{\partial^2 q}{\partial t^2}+2\lambda\frac{\partial q}{\partial t}-\frac{c^2}{2}\Bigl(\frac{\partial^2}{\partial u^2}+\frac{\partial^2}{\partial v^2}\Bigr)q =4\frac{\partial p_U}{\partial t}\frac{\partial p_V}{\partial t}-\frac{\partial^2 q}{\partial t^2}
\end{equation}
where the second-order planar telegraph different operator appearing in (\ref{equazioneStandard_uv}) emerges. By multiplying both members of (\ref{derivata2quv}) by the operator $\frac{\partial^2 }{\partial t^2}+2\lambda(t)\frac{\partial}{\partial t} +\lambda^2(t)$, we have
\begin{align}\label{operatore1_quv}
\Bigl(\frac{\partial^2 }{\partial t^2}+2\lambda\frac{\partial}{\partial t} +&\lambda^2\Bigr) \biggl( \frac{\partial^2 q}{\partial t^2}+2\lambda\frac{\partial q}{\partial t}-\frac{c^2}{2}\Bigl(\frac{\partial^2}{\partial u^2}+\frac{\partial^2}{\partial v^2}\Bigr)q \biggr) \\
&=\Bigl(\frac{\partial^2 }{\partial t^2}+2\lambda\frac{\partial}{\partial t} +\lambda^2\Bigr) \Bigl(4\frac{\partial p_U}{\partial t}\frac{\partial p_V}{\partial t}-\frac{\partial^2 q}{\partial t^2} \Bigr)\nonumber\\
& = \lambda^2 \Bigl(4\frac{\partial p_U}{\partial t}\frac{\partial p_V}{\partial t}-\frac{\partial^2 q}{\partial t^2} \Bigr)+ 2\lambda \Bigl(4\frac{\partial^2 p_U}{\partial t^2}\frac{\partial p_V}{\partial t}+4\frac{\partial p_U}{\partial t}\frac{\partial^2 p_V}{\partial t^2}-\frac{\partial^3 q}{\partial t^3} \Bigr)\nonumber\\
&\ \ \  +  \Bigl(4\frac{\partial^3 p_U}{\partial t^3}\frac{\partial p_V}{\partial t}+8\frac{\partial^2 p_U}{\partial t^2}\frac{\partial^2 p_V}{\partial t^2}+4\frac{\partial p_U}{\partial t}\frac{\partial^3 p_V}{\partial t^3}-\frac{\partial^4 q}{\partial t^4} \Bigr)\nonumber
\end{align}
where simple derivations are involved. In view of (\ref{CauchyTelegrafo}) applied successively and by convenient substitutions we produce the identity
\begin{align}\label{operatore2_quv}
\frac{c^4}{16}\Bigl(\frac{\partial^2}{\partial u^2}-\frac{\partial^2}{\partial v^2}\Bigr)^2q & = p_V\Bigl(\frac{\partial^4 p_U}{\partial t^4} +2\lambda\frac{\partial^3 p_U}{\partial t^3} +(\lambda^2+2\lambda')\frac{\partial^2 p_U}{\partial t^2}+(\lambda\lambda' +\lambda'')\frac{\partial p_U}{\partial t}\Bigr)\\
 &\ \ \ + p_U\Bigl(\frac{\partial^4 p_V}{\partial t^4} +2\lambda\frac{\partial^3 p_V}{\partial t^3} +(\lambda^2+2\lambda')\frac{\partial^2 p_V}{\partial t^2}+(\lambda\lambda' +\lambda'')\frac{\partial p_V}{\partial t}\Bigr)\nonumber\\
 &\ \ \ - 2\biggl( \frac{\partial^2 p_U}{\partial t^2}\frac{\partial^2 p_V}{\partial t^2}+\lambda^2\frac{\partial p_U}{\partial t}\frac{\partial p_V}{\partial t}+\lambda\Bigl(\frac{\partial^2 p_U}{\partial t^2}\frac{\partial p_V}{\partial t}+\frac{\partial p_U}{\partial t}\frac{\partial^2 p_V}{\partial t^2} \Bigr) \biggr)\nonumber.
\end{align}
By summing up (\ref{operatore1_quv}) and (\ref{operatore2_quv}) and by keeping in mind (\ref{derivata2quv0}), the terms with time-varying coefficients $\lambda^0,\lambda$ and $\lambda^2$ cancel out and we obtain
\begin{align}
\Bigl(\frac{\partial^2 }{\partial t^2}+2\lambda\frac{\partial}{\partial t} &+\lambda^2\Bigr) \biggl( \frac{\partial^2 q}{\partial t^2}+\lambda\frac{\partial q}{\partial t}-\frac{c^2}{2}\Bigl(\frac{\partial^2}{\partial u^2}+\frac{\partial^2}{\partial v^2}\Bigr)q \biggr) +\frac{c^4}{16}\Bigl(\frac{\partial^2}{\partial u^2}-\frac{\partial^2}{\partial v^2}\Bigr)^2q \label{termineDimostrazioneDecomposizione}\\
& = (\lambda\lambda'+\lambda'')\Bigl(\frac{\partial p_U}{\partial t} p_V+p_U
\frac{\partial p_V}{\partial t}\Bigr) + 2\lambda'\Bigl(\frac{\partial^2 p_U}{\partial t^2} p_V+p_U\frac{\partial^2 p_V}{\partial t^2}\Bigr) \nonumber \\
&=(\lambda\lambda'+\lambda'')\frac{\partial q}{\partial t}+ 2\lambda'\Bigl(\frac{\partial^2q}{\partial t^2}- 2\frac{\partial p_U}{\partial t}\frac{\partial p_V}{\partial t} \Bigr) \nonumber\\
& =(\lambda\lambda'+\lambda'')\frac{\partial q}{\partial t}+ 2\lambda'\frac{\partial^2q}{\partial t^2}-2\lambda'\biggl( \frac{\partial^2 q}{\partial t^2}+\lambda\frac{\partial q}{\partial t}-\frac{c^2}{4}\Bigl(\frac{\partial^2}{\partial u^2}+\frac{\partial^2}{\partial v^2}\Bigr)q \biggr) \nonumber
\end{align}
where in the last equality we suitably used (\ref{derivata2quv0}). Equation (\ref{termineDimostrazioneDecomposizione}) coincides with (\ref{equazioneStandard_uv}) and this completes the proof concerning the absolutely continuous component.
\end{proof}

In force of (\ref{decomposizioneXYUV}), assuming that $X(t)+Y(t)=z\in [-ct,ct]$, the distribution of the standard orthogonal motion coincides with that of a one-dimensional telegraph process with rate $\lambda(t)/2$ and velocities $\pm c$. It similarly happens if we assume $X(t)-Y(t)=z\in [-ct,ct]$. We point out that the displayed conditions imply that the motion lies on a segment that is parallel to some sides of $\partial S_{ct}$.
\\

Representation (\ref{decomposizioneXYUV}) permits us to focus on the study of one-dimensional processes only. We thus know the explicit distribution of $\bigl(X(t),Y(t)\bigr)$ when the rate function $\lambda(t),\ t>0,$ has one of the following forms, with $\lambda >0$,
\[\lambda(t)=\lambda,\ \ \ \ \ \lambda(t)=\frac{\lambda}{t} \ \text{(Foong and Van Kolck \cite{FVk1992})},\]
\[\lambda(t)=\lambda\text{th}(\lambda t) \ \text{(Iacus \cite{I2001})},\ \ \ \lambda(t)=\lambda\text{coth}(\lambda t) \ \text{(Garra and Orsingher \cite{GO2016})}.\]
For $\lambda(t) =\lambda/t,\: \lambda\text{coth}(\lambda t)$, the singular component of the distribution is absent because $\int_0^t \lambda(s)\dif s = \infty$.

\begin{remark}[$L_1$-distance]
Let $\{Z(t)\}_{t\ge0}$ be the stochastic process describing the Manhattan distance of the particle $\{\bigl(X(t),Y(t)\bigr)\}_{t\ge0}$ from the origin of the coordinate axes, i.e. $Z(t)= |X(t)|+|Y(t)|$. Let $\{U(t)\}_{t\ge0}$ and $\{V(t)\}_{t\ge0}$ be two independent one-dimensional telegraph processes with parameters $(c/2, \lambda(t)/2)$. By means of Theorem \ref{teoremaDecomposizioneXYUV} we have that, for $z\in [0,ct)$
\begin{align}\label{}
P\{Z(t)\in \dif z\}/\dif z &= \frac{\partial }{\partial z}P\{|X(t)|+|Y(t)|\le z\} =4\frac{\partial }{\partial z}\int_0^z\dif x\int_0^{z-x}p(x,y,t)\dif y \nonumber\\
&=2\frac{\partial }{\partial z}\int_0^z\dif x\int_0^{z-x}p_U\Bigl(\frac{x+y}{2},t\Bigr)p_V\Bigl(\frac{x-y}{2},t\Bigr)\dif y\nonumber\\
&= 2\,p_U\Bigl(\frac{z}{2},t\Bigr)\int_0^z p_V\Bigl(\frac{2x-z}{2},t\Bigr)\dif x=4\,p_U\Bigl(\frac{z}{2},t\Bigr)\int_0^{\frac{z}{2}} p_V(w,t)\dif w\nonumber.
\end{align}

We can also express the maximum $L_1$-distance of $\bigl(X(t),Y(t)\bigr)$ from the origin in terms of the distribution of a functional of a one-dimensional telegraph process. Let $\beta\in [0,ct]$, with (\ref{decomposizioneXYUV}) at hand, we obtain that
\begin{align}\label{}
P\Big\{&\max_{0\le s\le t}Z(s)\le \beta\Big\} = P\Big\{\max_{0\le s\le t}(\,|X(s)|+|Y(s)|\,)\le\beta\Big\}   \nonumber\\
&=P\Big\{\max_{0\le s\le t}\big\{\,X(s)+Y(s),\,-X(s)-Y(s),\,X(s)-Y(s),\,-X(s)+Y(s)\,\big\}\le\beta\Big\}\nonumber \\
&=P\Big\{\max_{0\le s\le t}\{2|U(s)|,\,2|V(s)|\}\le\beta\Big\}=P\Big\{\max_{0\le s\le t} |U(s)|\le\beta/2\Big\}^2.
 \nonumber
\end{align}
As far as we know, the distribution of $\max_{0\le s\le t} |U(s)|$ is unknown for all rate functions $\lambda(t)$. The interested reader can find a detailed study of the (one-sided) maximum of the constant rate one-dimensional telegraph process in Cinque and Orsingher \cite{CO2020}.
\hfill $\diamond$
\end{remark}

\begin{figure}
	\begin{minipage}{0.5\textwidth}
		\centering
		\begin{tikzpicture}[scale = 0.71]
		\draw[dashed, gray] (4,0) -- (0,4) node[above right, black, scale = 0.9]{$ct$};
		\draw[dashed, gray] (0,4) -- (-4,0) node[above left, black, scale = 0.9]{$-ct$};
		\draw[dashed, gray] (-4,0) -- (0,-4) node[below left, black, scale = 0.9]{$-ct$};
		\draw[dashed, gray] (0,-4) -- (4,0) node[above right, black, scale = 0.9]{$ct$};
		\draw (2.5,0) -- (0,2.5) node[right, black, scale = 0.8]{$z$};
		\draw(0,2.5) -- (-2.5,0) node[above, black, scale = 0.8]{$-z\ \ $};
		\draw (-2.5,0) -- (0,-2.5) node[left, black, scale = 0.8]{$-z$};
		\draw (0,-2.5) -- (2.5,0) node[above, black, scale = 0.8]{$z$};
		\draw[->, thick, gray] (-5,0) -- (5,0) node[below, scale = 1, black]{$\pmb{X(t)}$};
		\draw[->, thick, gray] (0,-5) -- (0,5) node[left, scale = 1, black]{ $\pmb{Y(t)}$};
		\draw (0,0)--(0.5,0)--(0.5,1)--(-0.3,1)--(-0.3,1.9)--(-1,1.9)--(-1,1.5);
		\filldraw (0,0) circle (0.8pt); \filldraw (0.5,0) circle (0.8pt); \filldraw (0.5,1) circle (1pt); \filldraw (-0.3,1) circle (0.8pt); \filldraw (-0.3,1.9) circle (0.8pt); \filldraw (-1,1.9) circle (0.8pt);
		\draw (0,0)--(-0.8,0)--(-0.8,-2.2)--(-1.2,-2.2)--(-1.2,-2.8);
		\filldraw (-0.8,0) circle (1pt); \filldraw (-0.8,-2.2) circle (0.8pt); \filldraw (-1.2,-2.2) circle (0.8pt); \filldraw (-1.2,-2.2) circle (0.8pt);
		\draw (0,0)--(0,-0.8)--(1,-0.8)--(1,-0.9)--(1.6,-0.9)--(1.6, 0.4);
		\filldraw(0,-0.8)circle (0.8pt);\filldraw(1,-0.8)circle (0.8pt);\filldraw(1,-0.9)circle (0.8pt);\filldraw(1.6,-0.9)circle (0.8pt);
		\end{tikzpicture}
		\caption{\small The trajectory starting with $d_0$ has $L_1$-distance from the origin equal to  $z\in$$(0,ct)$\newline at time $t$. The trajectory starting  with $d_3$ \newline has maximum $L_1$-distance from the origin \newline equal to $z$.}\label{L1dist}
	\end{minipage}\hfill
	\begin{minipage}{0.5\textwidth}
		\centering
		\begin{tikzpicture}[scale = 0.71]
		\draw[dashed, gray] (3,0) -- (0,4.5) node[right, black, scale = 0.9]{$c_Yt$};
		\draw[dashed, gray] (0,4.5) -- (-3,0) node[above left, black, scale = 0.9]{$-c_Xt$};
		\draw[dashed, gray] (-3,0) -- (0,-4.5) node[below left, black, scale = 0.9]{$-c_Yt$};
		\draw[dashed, gray] (0,-4.5) -- (3,0) node[above right, black, scale = 0.9]{$c_Xt$};
		\draw[->, thick, gray] (-5,0) -- (5,0) node[below, scale = 1, black]{$\pmb{X(t)}$};
		\draw[->, thick, gray] (0,-5.2) -- (0,5.2) node[left, scale = 1, black]{ $\pmb{Y(t)}$};
		\draw (0,0)--(-0.5,0)--(-0.5,1)--(-0.5,2.25)--(-1.5,2.25);
		\filldraw (-0.5,0) circle (1pt); \filldraw (-0.5,1) circle (0.8pt); \filldraw (-0.5,2.25) circle (0.8pt);
		\draw (0,0)--(0,0.3)--(1.2,0.3)--(1.2,0.6)--(1.2,1.18)--(0.8,1.18)--(0.5,1.18)--(0.15,1.18);
		\filldraw (0,0.3) circle (1pt); \filldraw (1.2,0.3) circle (0.8pt); \filldraw (1.2,0.6) circle (0.8pt); \filldraw (1.2,1.18) circle (0.8pt);\filldraw (0.8,1.18) circle (0.8pt); \filldraw (0.5,1.18) circle (0.8pt); 
		\draw (0,0)--(0.2,0)--(0.2,-1.3)--(0.2,-1.5)--(0.8,-1.5)--(0.8,-1.1)--(0.6,-1.1)--(0.6,-2.1);
		\filldraw (0.2,0) circle (1pt); \filldraw (0.2,-1.3) circle (0.8pt); \filldraw (0.2,-1.5) circle (0.8pt); \filldraw (0.8,-1.5) circle (0.8pt);\filldraw (0.8,-1.1) circle (0.8pt); \filldraw (0.6,-1.1) circle (0.8pt);
		\draw (0,0)--(0,-0.4)--(-0.8,-0.4)--(-0.8,-1)--(-0.5,-1)--(-0.5,-1.5)--(-0.5,-2.4)--(-0.65,-2.4)--(-0.8,-2.4);
		\filldraw (0,-0.4) circle (0.8pt); \filldraw (-0.8,-0.4) circle (0.8pt); \filldraw (-0.8,-1) circle (0.8pt);\filldraw (-0.5,-1) circle (1pt);  \filldraw (-0.5,-1.5) circle (0.8pt); \filldraw (-0.5,-2.4) circle (0.8pt); \filldraw (-0.65,-2.4) circle (0.8pt);
		\end{tikzpicture}
		\caption{\small Sample paths of an asymmetric ($c_X<c_Y$) standard orthogonal planar motion with Bernoulli trials.}\label{AsimmetricoBernStandard}
	\end{minipage}
\end{figure}

\begin{remark}[Standard motion with Bernoulli trials]\label{remarkStandardBernoulliLegge}
Let us consider the $q$-standard motion $\{\bigl(X_q(t),Y_q(t)\bigr)\}_{t\ge0}$, with $q\in(0,1]$ being the probability that the particle changes direction at Poisson times (see Remark \ref{remarkStandardBernoulliEquazione}). In this case, the particle, at time $t\ge0$, will be located in the square $S_{ct}$ defined in (\ref{supportoQuadrato}). However, if $\int_0^t \lambda(s)\dif s<\infty$, it can reach also the vertices of the square whatever the number of Poisson events in $[0,t]$ is,
\[P\big\{\bigl(X_q(t),Y_q(t)\bigr)\in V_{ct}\big\} = \sum_{n=0}^\infty(1-q)^nP\{N(t) = n\}  = e^{-q\int_0^t \lambda(s)\dif s}.\]

It is very interesting to observe that, by keeping in mind Remark \ref{remarkStandardBernoulliEquazione} and by proceeding as above, also this motion can be expressed by means of the representation (\ref{decomposizioneXYUV}), but in this case the one-dimensional telegraph processes are such that at all Poisson events, the change of direction occurs with probability $q\in(0,1]$. Now, by observing that such a telegraph process is equal in distribution to a telegraph process of rate $q\lambda(t)$, we obtain the following statement:
\\\textit{the standard motion with Bernoulli trials (with probability of changes equal to $q\in(0,1]$) is equal in distribution to a standard orthogonal motion with rate function $\lambda_q(t) = q\lambda(t)$}.
\end{remark}

\begin{remark}[Asymmetric motion]\label{remarkAsimmetricoLegge}
Let $\{\bigl(X(t),Y(t)\bigr)\}_{t\ge0}$ be the stochastic motion introduced in Remark \ref{remarkEquazioneAsimmetrico}, i.e. describing the movement of a particle running on the $x$-axis with velocity $c_X=|v_1|=|v_3|>0$ and on the $y$-axis with velocity $c_Y = |v_2|=|v_3|>0$. Then the following representation holds (in distribution)
\begin{equation}\label{decomposizioneXYUVasimmetrica}
\begin{cases}
X (t)= c_X\bigl(U(t)+V(t)\bigr)\\
Y(t) = c_Y\bigl(U(t)-V(t)\bigr)
\end{cases}
\end{equation}
where $\{U(t)\}_{t\ge0}$ and $\{V(t)\}_{t\ge0}$ are two independent one-dimensional telegraph processes with parameters $(1/2, \lambda(t)/2)$, with $\lambda(t)\in C^2\bigl((0,\infty),[0,\infty)\bigl)$.
Remark \ref{remarkEquazioneAsimmetrico} and Theorem \ref{teoremaDecomposizioneXYUV} lead to (\ref{decomposizioneXYUVasimmetrica}) concerning an asymmetric standard orthogonal planar random motion. 

In view of Remark \ref{remarkStandardBernoulliLegge}, all these considerations holds for the standard motion with Bernoulli trials as well.\hfill$\diamond$
\end{remark}

\subsection{The marginal component of the planar motion}

Thanks to decomposition (\ref{decomposizioneXYUV}), we can obtain the law of the marginal components of the vector process $\{\bigl(X(t),Y(t)\bigr)\}_{t\ge0}$, that is we can infer the distribution of the projection $X(t)$ (or equivalently $Y(t)$). The distribution of $X(t)$ follows from the convolution of two one-dimensional telegraph processes $U(t)$ and $V(t)$ with rates $(\lambda(t)/2, c/2)$. For $|x|<ct$
\begin{equation*}
P\{X(t)\in \dif x\} = p(x,t)\dif x = \frac{\dif x}{2}\int_{\max\{-\frac{ct}{2}+x, -\frac{ct}{2}\}}^{\min\{\frac{ct}{2}, \frac{ct}{2}+x\}} p_U(y,t)p_V(x-y,t)\dif y.
\end{equation*}
Alternatively, from (\ref{decomposizioneXYUV}) we can write
\begin{equation*}
P\{X(t)\in \dif x\} =  \frac{\dif x}{2}\int_{|x|-ct}^{-|x|+ct} p_U\Bigl(\frac{x+y}{2},t\Bigr)p_V\Bigl(\frac{x-y}{2},t\Bigr)\dif y.
\end{equation*}
We recall that the investigation of the sum of two independent telegraph process with constant rate function, $\lambda(t) = \lambda>0 \ \forall\ t$, has been carried out by Kolesnik \cite{K2014}.
\\

\begin{remark}
It is well-known that in the cases of $\lambda(t) = \lambda, \lambda\text{th}(\lambda t), \lambda\text{coth}(\lambda t)$, the distribution of the one-dimensional telegraph process is given in terms of the modified Bessel functions of order $0$ and $1$. It is interesting to show the following convolution
\begin{align}
&\int_{\max\{-\frac{ct}{2}+x, -\frac{ct}{2}\}}^{\min\{\frac{ct}{2}, \frac{ct}{2}+x\}} I_0\Bigl(\frac{\lambda}{c}\sqrt{\frac{c^2t^2}{4}-y^2}\Bigr)  I_0\Bigl(\frac{\lambda}{c}\sqrt{\frac{c^2t^2}{4}-(y-x)^2}\Bigr)\Bigg|_{x = 0}\dif y\nonumber\\ 
&\ \ \ \ = \sum_{h=0}^\infty\sum_{k=0}^\infty \Bigl(\frac{\lambda}{2c}\Bigr)^{2h+2k} \frac{1}{h!^2k!^2}\int_{-\frac{ct}{2}}^{\frac{ct}{2}}\Bigl(\frac{c^2t^2}{4}-y^2\Bigr)^{h+k}\dif y \nonumber\\
&\ \ \ \  = \sum_{h=0}^\infty\sum_{k=0}^\infty \Bigl(\frac{\lambda}{2c}\Bigr)^{2h+2k} \frac{1}{h!^2k!^2} \Bigl(\frac{c t}{2}\Bigr)^{2h+2k+1}\frac{\Gamma(h+k+1)\Gamma(1/2)}{\Gamma(h+k+1+1/2)}  \label{passaggioCalcoloConvoluzioneI0in0}\\
&\ \ \ \  = c \sum_{h=0}^\infty\sum_{k=0}^\infty  \Bigl(\frac{\lambda t}{2}\Bigr)^{2h+2k} \frac{t\, (h+k)!^2}{h!^2k!^2(2h+2k+1)!} = c \sum_{h=0}^\infty\sum_{l=h}^\infty  \Bigl(\frac{\lambda t}{2}\Bigr)^{2l} \frac{t \,l!^2}{h!^2(l-h)!^2(2l+1)!} \nonumber\\
&\ \ \ \ = c \sum_{l=0}^\infty\Bigl(\frac{\lambda t}{2}\Bigr)^{2l} \frac{t }{(2l+1)!}\sum_{h=0}^l  \binom{l}{h}^{2}   = c \sum_{l=0}^\infty\Bigl(\frac{\lambda t}{2}\Bigr)^{2l} \frac{t }{(2l+1)!}\binom{2l}{l}\label{passaggio2CalcoloConvoluzioneI0in0} \\
&\ \ \ \ =c \sum_{l=0}^\infty\frac{1}{l!^2}\Bigl(\frac{\lambda }{2}\Bigr)^{2l} \int_0^t s^{2l}\dif s  = c\int_0^t I_0(\lambda s) \dif s \nonumber
\end{align}
where in step (\ref{passaggioCalcoloConvoluzioneI0in0}) we used the duplication formula of the Gamma function and in (\ref{passaggio2CalcoloConvoluzioneI0in0}) we applied the Vandermonde identity.
Similarly, we obtain
$$ \int_{\max\{-\frac{ct}{2}+x, -\frac{ct}{2}\}}^{\min\{\frac{ct}{2}, \frac{ct}{2}+x\}} I_0\Bigl(\frac{\lambda}{c}\sqrt{\frac{c^2t^2}{4}-y^2}\Bigr)  \frac{\partial }{\partial t}I_0\Bigl(\frac{\lambda}{c}\sqrt{\frac{c^2t^2}{4}-(y-x)^2}\Bigr)\Bigg|_{x = 0}\dif y =\frac{c}{2}\Bigl( I_0(\lambda t)-1\Bigl)$$
 and
$$ \int_{\max\{-\frac{ct}{2}+x, -\frac{ct}{2}\}}^{\min\{\frac{ct}{2}, \frac{ct}{2}+x\}} \frac{\partial }{\partial t}I_0\Bigl(\frac{\lambda}{c}\sqrt{\frac{c^2t^2}{4}-y^2}\Bigr)  \frac{\partial }{\partial t}I_0\Bigl(\frac{\lambda}{c}\sqrt{\frac{c^2t^2}{4}-(y-x)^2}\Bigr)\Bigg|_{x = 0}\dif y =\frac{\lambda c}{4}\Bigl(I_1(\lambda t)-\frac{\lambda t}{2}\Bigr).$$
We note that the density of the marginal $X(t)$ of the planar standard motion attains its maximum at $x = 0$.
\hfill$\diamond$
\end{remark}

Let $\{\bigl(X(t),Y(t)\bigr)\}_{t\ge0}$ be a standard orthogonal planar motion. We note that when the vector process moves horizontally, that is when $X(t)$ is active, the vertical motion is suspended up to the occurrence of the next Poisson event. When the vector proceeds along the vertical direction and a Poisson event occurs, the particle uniformly switches to either the rightward or leftward direction.
\\

Below we rigorously prove that the marginal component of the standard orthogonal planar motion is distributed as a one-dimensional telegraph-type process with three velocities, $\pm c$ ($c>0$) and $0$ and such that it changes speed at Poisson paced times with the following rule: if the current speed is $c$ or $-c$, then it stops, meaning that it changes to speed $0$; otherwise, if the velocity is $0$, it uniformly selects the next velocity between $c$ and $-c$.

\begin{theorem}\label{teoremaMotoConStop}
Let $\{N(t)\}_{t\ge0}$ be a Poisson process with rate function $\lambda(t)\in C^1\bigl((0,\infty),[0,\infty)\bigl)$, $V_0$ a r.v., independent of $N(t)$, and $\{V(t)\}_{t\ge0}$ such that $V(0) = V_0$,
\[V_0 = \begin{cases}\begin{array}{r l}
+c, & w.p.\ 1/4,\\0, & w.p.\ 1/2,\\-c, & w.p.\ 1/4,
\end{array}\end{cases} \text{and} \ \ \begin{cases}
P\{V(t+\dif t)=c\,|\,V(t) = 0,\,N(t,t+\dif t]=1\}=1/2,\\
P\{V(t+\dif t)=-c\,|\,V(t) = 0,\,N(t,t+\dif t]=1\}=1/2,\\
P\{V(t+\dif t)=0\,|\,V(t) = \pm c,\,N(t,t+\dif t]=1\}=1.
\end{cases} \]
Let $\{\mathcal{T}(t)\}_{t\ge0}$ be a one-dimensional stochastic motion such that
$\mathcal{T}(t) = \int_0^t V(s)\dif s$, then the transition density $p(x,t)\dif x = P\{\mathcal{T}(t)\in \dif x\}$ satisfies the following third-order partial differential equation
\begin{equation}\label{equazioneLeggeMarginaleStandard}
\frac{\partial^3p}{\partial t^3} +3\lambda(t)\frac{\partial p^2}{\partial t^2} +\bigl(2\lambda^2(t)+\lambda'(t) \bigr) \frac{\partial p}{\partial t}= c^2\frac{\partial^3 p}{\partial t\partial x^2}+c^2 \lambda(t)\frac{\partial^2 p}{\partial x^2}.
\end{equation}
Finally, if $\int_0^t\lambda(s)\dif s<\infty$,
\begin{equation}\label{singolaritaMarginale}
P\{\mathcal{T}(t) = ct\}=P\{\mathcal{T}(t) = -ct\}=\frac{1}{2}P\{\mathcal{T}(t) = 0\}=\frac{e^{-\int_0^t\lambda(s)\dif s}}{4}.
\end{equation}
\end{theorem}
We say that $\mathcal{T}$ has velocity $c$ and rate function $\lambda(t)$ (both positive).

\begin{proof}
Let $v_0\in\{-c,0,c\}$, then $P\{\mathcal{T}(t) = v_0t\} = P\{V_0 = v_0,\ N(t) = 0\}$ and this proves (\ref{singolaritaMarginale}).

To prove (\ref{equazioneLeggeMarginaleStandard}) we use the following probability functions
\begin{equation}
\begin{split}\label{leggiMotoUnidimensionaleCongiuntoV}
f_1(x,t)\dif x=P\{\mathcal{T}&(t)\in \dif x,\,V(t) = c\},\ \ \ f_2(x,t)\dif x =P\{\mathcal{T}(t)\in \dif x,\, V(t) = -c\},\\
& f_0(x,t)\dif x=P\{\mathcal{T}(t)\in \dif x, \,V(t)=0\}.
\end{split}\end{equation}
The probabilities (\ref{leggiMotoUnidimensionaleCongiuntoV}) are related by the following differential system of differential equations
\begin{equation}\label{sistemaMotoUnidimensionaleCongiuntoV}
\begin{cases}
\frac{\partial f_1}{\partial t} = -c\frac{\partial f_1}{\partial x}+\frac{\lambda(t)}{2}(f_0-2f_1)\\
\frac{\partial f_0}{\partial t} =\lambda(t)(f_1+f_2-f_0)\\
\frac{\partial f_2}{\partial t} = c\frac{\partial f_2}{\partial x}+\frac{\lambda(t)}{2}(f_0-2f_2)
\end{cases} \text{and then}\ \ \ \begin{cases}
\frac{\partial g_1}{\partial t} = -c\frac{\partial g_2}{\partial x}+\lambda(t) (f_0-g_1)\\
\frac{\partial g_2}{\partial t} = -c\frac{\partial g_1}{\partial x}-\lambda(t) g_2\\
\frac{\partial f_0}{\partial t} = \lambda(t)(g_1-f_0)
\end{cases}\end{equation}
where we simplified the first differential system by means of the auxiliary functions $g_1 = f_1+f_2,\ g_2 = f_1-f_2$. By suitably using the equations of the second system of (\ref{sistemaMotoUnidimensionaleCongiuntoV}) we pass to the differential system
\begin{equation}\label{sistemaMotoUnidimensionaleG2}
\begin{cases}
\frac{\partial^2 g_1}{\partial t^2} = c^2\frac{\partial^2 g_1}{\partial x^2}- \lambda(t)\frac{\partial g_1}{\partial t}+\lambda^2(t) (f_0-g_1)+\frac{\partial }{\partial t}\bigl( \lambda(t) (f_0-g_1)\bigr),\\
\frac{\partial f_0}{\partial t} = \lambda(t)(g_1-f_0).
\end{cases}
\end{equation}
By deriving the second equation with respect to $x$ and by considering the functions $p= g_1+f_0$ (this is the probability density of the motion) and $w = g_1-f_0$, after some calculation, we obtain
\begin{equation}\label{sistemaMotoUnidimensionaleP}
\begin{cases}
2\frac{\partial^2 p}{\partial t^2} = c^2\frac{\partial^2 p}{\partial x^2}+c^2\frac{\partial^2 w}{\partial x^2}- \lambda(t)\frac{\partial p}{\partial t}- \lambda(t)\frac{\partial w}{\partial t}-2\lambda^2(t)w =  c^2\frac{\partial^2 p}{\partial x^2}+c^2\frac{\partial^2 w}{\partial x^2}- 2\lambda(t)\frac{\partial p}{\partial t}, \\
\frac{\partial p}{\partial t} = \frac{\partial w}{\partial t}+2\lambda(t)w.
\end{cases}
\end{equation}
Finally, by deriving twice with respect to $t$ the second equation of (\ref{sistemaMotoUnidimensionaleP}) and substituting $\frac{\partial^2 w}{\partial x^2}$ from the first equation, we obtain the third order differential equation (\ref{equazioneLeggeMarginaleStandard}).
\end{proof}

\begin{theorem}\label{teoremaSommaTelegrafi}
Let $\{U(t)\}_{t\ge0}$, $\{V(t)\}_{t\ge0}$ be two independent one-dimensional telegraph processes with parameters $(c/2,\lambda(t)/2)$, with $\lambda(t)\in C^1\bigl((0,\infty),[0,\infty)\bigl)$. The process $X(t) = U(t)+V(t)$, $t\ge0$, is equal in distribution to the one-dimensional process $\mathcal{T}(t),\ t\ge0$, defined in Theorem \ref{teoremaMotoConStop}.
\end{theorem}

\begin{proof}
If $\int_0^t\lambda(s)\dif s<\infty$, it is straightforward to show that the probability mass of the discrete component of $X(t)=U(t)+V(t)$ coincides with (\ref{singolaritaMarginale}).

We now show that the absolutely continuous component of $X(t)$ satisfies the third-order partial differential equation (\ref{equazioneLeggeMarginaleStandard}).
$p_U(u,t)\dif u = P\{U(t)\in \dif u\}$ and $p_U(v,t)\dif v = P\{V(t)\in \dif v\}$ are both solutions of (\ref{CauchyTelegrafo}). By writing $H_X(\gamma, t) =\mathbb{E}\bigl[e^{i\gamma X(t)}\bigr]$ and $H_U(\gamma, t)=\mathbb{E}\bigl[e^{i\gamma U(t)}\bigr]$, $\gamma \in \mathbb{R}$, we readily obtain
\[  H_X(\gamma, t) = H_U^2(\gamma, t)\ \ \ \text{and}\ \ \ \frac{\partial^2 H_U}{\partial t^2}+ \lambda(t)\frac{\partial H_U}{\partial t} = -\frac{\gamma^2c^2}{4}H_U.\]
Thus,
\begin{equation*}\label{}
\frac{\partial H_X}{\partial t}=2H_U(\gamma, t)\frac{\partial H_U}{\partial t},\ \ \ \frac{\partial^2 H_X}{\partial t^2}=2 \Bigl(\frac{\partial H_U}{\partial t}\Bigr)^2-\lambda(t)\frac{\partial H_X}{\partial t}  -\frac{\gamma^2c^2}{2}H_X
\end{equation*}
and 
\begin{align*}
\frac{\partial^3 H_X}{\partial t^3}&=4\frac{\partial H_U}{\partial t}\frac{\partial^2 H_U}{\partial t^2}-\lambda'(t)\frac{\partial H_X}{\partial t}-\lambda(t)\frac{\partial^2 H_X}{\partial t^2}-\frac{\gamma^2c^2}{2}\frac{\partial H_X}{\partial t}\nonumber\\
&=-4\lambda(t) \Bigl(\frac{\partial H_U}{\partial t}\Bigr)^2-\gamma^2c^2H_U\frac{\partial H_U}{\partial t}-\lambda'(t)\frac{\partial H_X}{\partial t}-\lambda(t)\frac{\partial^2 H_X}{\partial t^2}-\frac{\gamma^2c^2}{2}\frac{\partial H_X}{\partial t} \nonumber\\
&=-2\lambda(t)\Bigl(\frac{\partial^2 H_X}{\partial t^2}+ \lambda(t)\frac{\partial H_X}{\partial t}  +\frac{\gamma^2c^2}{2}H_X\Bigr)-\gamma^2c^2\frac{\partial H_X}{\partial t} -\lambda'(t)\frac{\partial H_X}{\partial t}-\lambda(t)\frac{\partial^2 H_X}{\partial t^2}\nonumber\\
&=-3\lambda(t)\frac{\partial^2 H_X}{\partial t^2}-\bigl(2\lambda^2(t)+\lambda'(t)\bigr)\frac{\partial H_X}{\partial t}-\gamma^2c^2\frac{\partial H_X}{\partial t}-\lambda(t)\gamma^2c^2H_X
\end{align*}
and the inverse Fourier transform straightforwardly yields equation (\ref{equazioneLeggeMarginaleStandard}).
\end{proof}

By taking into account representation (\ref{decomposizioneXYUV}) the next statement follows as a consequence of the previous theorem.

\begin{corollary}
Let $\{\bigl(X(t),Y(t)\bigr)\}_{t\ge0}$ be a standard orthogonal planar motion with rate function $\lambda(t)\in C^2\bigl((0,\infty),[0,\infty)\bigl)$. The marginal processes $X(t)$ and $Y(t)$ are equal in distribution to the one-dimensional process $\mathcal{T}(t),\ t\ge0$, defined in Theorem \ref{teoremaMotoConStop}.
\end{corollary}

By suitably manipulating (\ref{equazioneLeggeMarginaleStandard}), we obtain that the absolutely continuous component of the distribution of the marginal component of the standard orthogonal motion $X(t) = U(t)+V(t)$ and $Y(t) = U(t)-V(t),\ t\ge0,$ satisfies
\begin{equation}\label{equazioneLeggeMarginaleStandardCostante}
\Bigl(\frac{\partial }{\partial t} +\lambda(t) \Bigr) \Bigl(\frac{\partial^2}{\partial t^2}+2\lambda(t)\frac{\partial}{\partial t}-c^2\frac{\partial^2}{\partial x^2} \Bigr)p=\lambda'(t)\frac{\partial p}{\partial t}.
\end{equation}
Equation (\ref{equazioneLeggeMarginaleStandardCostante}) reduces to formula (4.7) of Kolesnik \cite{K2014} if $\lambda(t) = \lambda \ \forall\ t$ (we recall that we are considering $U$ and $V$ independent telegraph processes with velocity $c/2$ and rate function $\lambda(t)/2$).

\begin{remark}[Asymmetric motion]
By keeping in mind formula (\ref{decomposizioneXYUVasimmetrica}), thanks to Theorem \ref{teoremaSommaTelegrafi}, we obtain that the marginal components $X(t)$ and $Y(t)$, $t\ge0$, of the asymmetric standard planar motion are equal in distribution to the process $\mathcal{T}(t),\ t\ge0$, defined in Theorem \ref{teoremaMotoConStop}, with velocities $c_X$ and $c_Y$ respectively. \hfill$\diamond$
\end{remark}

\section{Reflecting orthogonal planar random motions}\label{Sezione3}

We consider $\{\bigl(X(t),Y(t)\bigr)\}_{t\ge0}$ a \textit{reflecting orthogonal planar random motion} with changes of direction paced by a non-homogeneous Poisson process $\{N(t)\}_{t\ge0}$ with rate $\lambda(t)\in C^1\bigl((0,\infty),[0,\infty)\bigl)$. $\bigl(X(t),Y(t)\bigr)$ is the stochastic vector process describing the position, at time $t$, of a particle moving with the rules of the standard orthogonal motion (see above), but with the possibility to bounce back. Therefore, at all Poisson times it can uniformly switch to one of the available directions, except for the current one. This motion has been studied by Kolesnik and Orsingher \cite{KO2001} in the case of a constant rate function $\lambda(t) = \lambda>0\ \forall \ t$.
\\

At time $t\ge0$, the reflecting orthogonal motion is located in the square $S_{ct}$ defined in (\ref{supportoQuadrato}). If the rate function  is such that $\Lambda(t) = \int_0^t\lambda(s) \dif s<\infty$, the distribution of the motion has two singular components, the border of the square $S_{ct}$ and its diagonals. Let $V_{ct} = \{(0,\pm ct),(\pm ct,0)\}$ being the set of the vertices of $S_{ct}$ and $\partial D_{ct} = \{(x,y)\in S_{ct}\,:\, x=0\ \text{or}\ y=0\}$ being the diagonals of $S_{ct}$, then $P\{\bigl(X(t),Y(t)\bigr) \in V_{ct}\}=e^{-\Lambda(t)}$,
\begin{equation}\label{riflessioneProbFrontiera}
P\{\bigl(X(t),Y(t)\bigr) \in \partial S_{ct}\setminus V_{ct}\}=\frac{2}{3}\sum_{n=1}^\infty P\{N(t) = n\}\frac{1}{3^{n-1}} = 2\Bigl(e^{-\frac{2\Lambda(t)}{3}}-e^{-\Lambda(t)}\Bigr),
\end{equation}
because, in order to reach the edge $\partial S_{ct}$, the particle can not bounce back and it must always move towards the ``selected'' side of the border. Finally,
\begin{equation}\label{riflessioneProbCroce}
P\{\bigl(X(t),Y(t)\bigr) \in \partial D_{ct}\setminus V_{ct}\}=\sum_{n=1}^\infty P\{N(t) = n\}\frac{1}{3^{n}} =  e^{-\frac{2\Lambda(t)}{3}}-e^{-\Lambda(t)},
\end{equation}
because the particle must always move back and forth.

We now present a general result concerning the probability density inside the square $S_{ct}$, that is $p(x,y,t)\dif x\dif y = P\{X(t)\in \dif x,\ Y(t)\in \dif y\}$, for $(x,y)\in S_{ct}$.

\begin{theorem}
\label{teoremaEquazioneRiflessione}
The absolutely continuous component $p = p(x,y,t)\in C^4\bigl(\mathbb{R}^2\times[0,\infty), [0,\infty)\bigr)$ of the distribution of the reflecting orthogonal process $\{\bigl(X(t),Y(t)\bigr)\}_{t\ge0}$ satisfies the following fourth-order differential system
\begin{equation}\label{equazioneRiflessione}
\begin{cases}

\frac{\partial^4p}{\partial t^4} + 4\lambda \frac{\partial^3p}{\partial t^3}+\Bigl(\frac{16}{3}\lambda^2+4\lambda'\Bigr)\frac{\partial^2p}{\partial t^2}+ \Bigl(\frac{64}{27}\lambda^3+\frac{16}{3}\lambda\lambda'+\frac{4}{3}\lambda''\Bigr)\frac{\partial p}{\partial t} + c^4\frac{\partial^4 p}{\partial x^2\partial y^2}\\
\hspace{5cm}=c^2\Bigl(\frac{\partial^2}{\partial x^2}+\frac{\partial^2}{\partial y^2}\Bigr)\Bigl(\frac{\partial^2p}{\partial t^2}+2\lambda\frac{\partial p}{\partial t} + \bigl(\frac{8}{9}\lambda^2+\frac{2}{3}\lambda'\bigr)p\Bigr),\\
p(x,y,t)\ge0,\\
\int_{S_{ct}}p(x,y,t)\dif x\dif y = 1- 3e^{-\frac{2}{3}\int_0^t\lambda(s)\dif s}+2e^{-\int_0^t\lambda(s)\dif s},
\end{cases}\end{equation}
where $\lambda=\lambda(t)\in C^{2}\bigl((0,\infty),[0,\infty)\bigl)$ denotes the rate function of the non-homogeneous Poisson process governing the changes of direction.
\end{theorem}

\begin{proof}
By means of notation (\ref{notazione}), we obtain the differential system
\begin{equation}\label{sistemaRiflessione}
\begin{cases}
\frac{\partial f_0}{\partial t} = -c\frac{\partial f_0}{\partial x}+\frac{\lambda(t)}{3}(f_1+f_2+f_3-3f_0),\\
\frac{\partial f_1}{\partial t} = -c\frac{\partial f_1}{\partial y}+\frac{\lambda(t)}{3}(f_0+f_2+f_3-3f_1),\\
\frac{\partial f_2}{\partial t} = c\frac{\partial f_2}{\partial x}+\frac{\lambda(t)}{3}(f_0+f_1+f_3-3f_2),\\
\frac{\partial f_3}{\partial t} = c\frac{\partial f_3}{\partial y}+\frac{\lambda(t)}{3}(f_0+f_1+f_2-3f_3).
\end{cases}\end{equation}
The claimed result follows by proceeding as in Section \ref{sottoSezioneEquazioneStandard} in order to proof Theorem \ref{teoremaEquazioneStandard}.
\end{proof}

By dividing the differential equation (\ref{equazioneRiflessione}) by $\lambda^3$, it is easy to observe that the reflecting orthogonal planar motion converges to Brownian motion with diffusivity $3\sigma^2/4$ under Kac's conditions.
\\

We now consider $\int_0^t\lambda(s)\dif s<\infty$ and we study the density on the singular components, i.e. on the border and the diagonals of the square $S_{ct}$.

\begin{proposition}\label{proposizioneFrontieraRiflessione}
Let $\{\bigl(X(t),Y(t)\bigr)\}_{t\ge0}$ be a reflecting orthogonal planar random motion with changes of direction paced by a Poisson process with rate $\lambda(t)\in C^1\bigl((0,\infty),[0,\infty)\bigr)$ such that $\Lambda(t) = \int_0^t\lambda(s)\dif s<\infty$ for all $t>0$. Then, $f(\eta, t)\dif \eta =P\{X(t)+Y(t) = ct, X(t)-Y(t)\in \dif \eta\}$, for $|\eta|<ct$, satisfies the differential problem
\begin{equation}\label{sistemaFrontieraAltoDxRiflessione}
\begin{cases}
\frac{\partial^2 f}{\partial t^2} +2\lambda(t)\frac{\partial f}{\partial t} +\frac{2}{3}\Bigl(\frac{4}{3}\lambda(t)^2+\lambda'(t)\Bigr)f =c^2\frac{\partial^2 f}{\partial \eta^2},\\
f(\eta,t)\ge0,\\
\int_{-ct}^{ct}f(\eta,t)\dif \eta = \frac{1}{2}(e^{-2\Lambda(t)/3}-e^{-\Lambda(t)}).
\end{cases}
\end{equation}
The function
\begin{equation} \label{formulaGeneraleFrontieraRiflessione}
f(\eta, t) = \frac{e^{-2\Lambda(t)/3}}{4} \bar p\Bigl(\frac{\eta}{2},t\Bigr) ,
\end{equation}
 solves (\ref{sistemaFrontieraAltoDxRiflessione}). In (\ref{formulaGeneraleFrontieraRiflessione}), $\bar p(x,t)$ is the absolutely continuous component of the distribution of a telegraph process with rate function $\lambda(t)/3$ and velocity $c/2$.
\end{proposition}

The proposition refers to the probability that the motion lies on the border of its support in the first quadrant, defined in  (\ref{insiemeFrentieraAltoDx}). Equivalent results hold for the other three sides of $\partial S_{ct}$.

\begin{proof}
In order to obtain the differential equation in (\ref{sistemaFrontieraAltoDxRiflessione}) we proceed as in Proposition \ref{propFrontieraStandard}. By using notation (\ref{notazioneFrontieraStandard}) we can write
\begin{equation}\label{sistemiFrontieraRiflessioneG}
\begin{cases}
\frac{\partial f_0}{\partial t} = -c\frac{\partial f_0}{\partial \eta}+\frac{\lambda(t)}{3}(f_1-3f_0)\\
\frac{\partial f_1}{\partial t} = c\frac{\partial f_1}{\partial \eta}+\frac{\lambda(t)}{3}(f_0-3f_1)
\end{cases}\text{and}\ \ \ 
\begin{cases}
\frac{\partial f}{\partial t} = -c\frac{\partial w}{\partial \eta}-\frac{2\lambda(t)}{3}f\\
\frac{\partial w}{\partial t} = -c\frac{\partial f}{\partial \eta}-\frac{4\lambda(t)}{3}w
\end{cases}
\end{equation}
where we performed the change of variables $f = f_0+f_1,\ w=f_0-f_1$. Now, it is easy to obtain the first formula of system (\ref{sistemaFrontieraAltoDxRiflessione}).

To prove the second part of the statement, we proceed as follows. Let $\{\mathcal{T}_c(t)\}_{t\ge0}$, $c>0$, denotes a one-dimensional telegraph process with rate function $\lambda(t)/3$ and velocities $\pm c$. From the definition of the telegraph motion we obtain that $\mathcal{T}_c(t)/2 \stackrel{d}{=} \mathcal{T}_{c/2}(t)\ \forall\ t$. Now, we express (\ref{formulaGeneraleFrontieraRiflessione}) in terms of the density $q$ of $\mathcal{T}_c(t
)$,
$$q(\eta, t)\dif \eta = P\{\mathcal{T}_c(t)\in \dif \eta\} =  P\Big\{\mathcal{T}_{c/2}(t)\in \frac{\dif \eta}{2}\Big\} = \frac{\dif \eta}{2}\bar p\Bigl(\frac{\eta}{2},t\Bigr) $$
thus $f(\eta,t) = q(\eta, t)e^{-2\Lambda(t)/3}$. By keeping in mind that $q$ satisfies the generalized telegraph equation 
$$\frac{\partial^2 q}{\partial t^2} +\frac{2\lambda(t)}{3}\frac{\partial q}{\partial t} = c^2\frac{\partial^2 q}{\partial \eta^2},\ \text{ with condition }\ \int_{-ct}^{ct}q(\eta, t)\dif \eta = 1-e^{-\frac{\Lambda(t)}{3}},$$ 
it is easy to show that (\ref{formulaGeneraleFrontieraRiflessione}) satisfies (\ref{sistemaFrontieraAltoDxRiflessione}).
\end{proof}

Proposition \ref{proposizioneFrontieraRiflessione} concerns only the side $\partial S^{(1)}_{ct} = \{(x,y)\in \mathbb{R}^2\,:\, x+y = ct,\\ |x-y|\le ct\}$, but it equivalently holds on all the other components of the border $\partial S_{ct}$.

We can interpret the above proposition thanks to the following reasoning. Inspired by the results of Theorem \ref{teoremaDecomposizioneXYUV} we consider the rotated process $\{\bigl(U(t),V(t)\bigr)\}_{t\ge0}$ where
\begin{equation}\label{decomposizioneXYUVRiflessione}
U(t) = \frac{X (t)+Y(t)}{2},\ \ \ V(t) = \frac{X (t)-Y(t)}{2}.
\end{equation}
By means of a direct investigation of the marginal process $U$ (and equivalently $V$) we observe that it describes a telegraph motion with velocity $c/2$ and rate function $2\lambda(t)/3$. In fact, at each Poisson event, the projection $U(t)$ changes direction with probability $2/3$, due to the switch of the motion $\bigl(X(t),Y(t)\bigr)$ (for instance, if $\bigl(X(t),Y(t)\bigr)$ moves with velocity $D(t)=d_0$, $U(t)$ moves with speed $+c/2$ and it will change velocity if $\bigl(X(t),Y(t)\bigr)$ takes either direction $d_1$ or $d_2$).
Therefore, $U$ and $V$ are identical, but dependent, telegraph processes with rate function $2\lambda(t)/3$ and velocity $c/2$.

Now, in view of Proposition \ref{proposizioneFrontieraRiflessione} and (\ref{decomposizioneXYUVRiflessione}), we obtain that 
\begin{align*}
f(\eta,t)\dif \eta &= P\{X(t)+Y(t)=ct,\ X(t)-Y(t)\in \dif \eta\} \\
&= P\Big\{U(t)=\frac{ct}{2}\Big\}P\Big\{ V(t)\in \frac{\dif \eta}{2}\,\Big|\,U(t) = \frac{ct}{2} \Big\} = \frac{e^{-2\Lambda(t)/3}}{2} \frac{\bar p\bigl(\frac{\eta}{2},t\bigr)}{2},
\end{align*}
meaning that $V(t)$, knowing that $U(t)= ct/2$, is distributed as a telegraph process with rate function $\lambda(t)/3$ and velocity $c/2$. 


\begin{remark}
In the particular case $\lambda(t) = \lambda\ \forall\ t$, we can solve system (\ref{sistemaFrontieraAltoDxRiflessione}) by using the transformation $f(\eta,t)=e^{-\lambda t}q(\eta, t)$. This leads to the Klein-Gordon equation $\frac{\partial^2 q}{\partial t^2} -\frac{c^2}{4}\frac{\partial^2 q}{\partial \eta^2}=\frac{\lambda^2}{9}q$, $|\eta|<ct$. Therefore, the solution of  (\ref{sistemaFrontieraAltoDxRiflessione}) is
\begin{align}
f(\eta, t)&=\frac{e^{-\lambda t}}{4c}\Biggl[\frac{\lambda}{3} I_0\Bigl(\frac{\lambda}{3c}\sqrt{c^2t^2-\eta^2} \Bigr)+\frac{\partial}{\partial t}I_0\Bigl(\frac{\lambda}{3c}\sqrt{c^2t^2-\eta^2} \Bigr)\Biggr]\label{frontieraRiflessioneCostante} \\
&=\frac{e^{-\frac{2\lambda t}{3}}}{2}\,\frac{e^{-\frac{\lambda t}{3}}}{2c}\Biggl[\frac{\lambda}{3} I_0\Bigl(\frac{\lambda}{3c}\sqrt{c^2t^2-\eta^2} \Bigr)+\frac{\partial}{\partial t}I_0\Bigl(\frac{\lambda}{3c}\sqrt{c^2t^2-\eta^2} \Bigr)\Biggr]. \nonumber
\end{align} 
which coincides with (\ref{formulaGeneraleFrontieraRiflessione}). Probability (\ref{frontieraRiflessioneCostante}) first appeared in Kolesnik and Orsingher (2001) (formula (3.3)). Furthermore, we derive the distribution conditionally on the exact number of changes of direction, for integer $k\ge0$,
\begin{align}
&P\{X(t)+Y(t) = ct, X(t)-Y(t)\in \dif \eta\,|\,N(t)=2k+1\} \nonumber\\
&=3P\{X(t)+Y(t) = ct, X(t)-Y(t)\in \dif \eta\,|\,N(t)=2k+2\}=\frac{(2k+1)!}{2\,k!^2}\frac{(c^2t^2-\eta^2)^k}{(6ct)^{2k+1}}\dif \eta \nonumber
\end{align}
where $|\eta|<ct$.\hfill $\diamond$
\end{remark}

\begin{proposition}\label{propCroceRiflessione}
Let $\{\bigl(X(t),Y(t)\bigr)\}_{t\ge0}$ be a reflecting orthogonal planar random motion with changes of direction paced by a Poisson process with rate $\lambda(t)$ such that $\Lambda(t) = \int_0^t\lambda(s)\dif s<\infty$ for all $t>0$. Then, for $|x|<ct$, $f(x,t) = P\{X(t)\in \dif x, Y(t)=0\}$, satisfies the differential problem (\ref{sistemaFrontieraAltoDxRiflessione}) with $x$ instead of $\eta$ and it is given by (\ref{formulaGeneraleFrontieraRiflessione}).
\end{proposition}

The proposition concerns the distribution on the horizontal diagonal of $S_{ct}$, but it equivalently holds on the vertical one $\{(x,y)\in S_{ct}\,:\,x=0\}$.

\begin{proof}
The probabilities
\begin{equation*}
\begin{cases}
g_0(x, t) \dif \eta = P\{X(t)\in \dif x,\ Y(t)=0,\ D(t) = d_0\}\\
g_2(x, t) \dif \eta = P\{X(t)\in \dif x,\ Y(t)=0,\ D(t) = d_2\}
\end{cases}
\end{equation*}
satisfy the systems appearing in (\ref{sistemiFrontieraRiflessioneG}) with $f_0\rightarrow g_0,\,f_1\rightarrow g_2$ and $f = g_0+g_2,\, w = g_0-g_2$.
\end{proof}

We now focus on the marginal component of the reflecting motion.

\begin{theorem}
Let $\{\bigl(X(t),Y(t)\bigr)\}_{t\ge0}$ be a reflecting orthogonal planar motion whose changes of direction are governed by a Poisson process $\{N(t)\}_{t\ge0}$ with rate function $\lambda(t)\in C^1\bigl((0,\infty),[0,\infty)\bigl)$. The absolutely continuous component of the distribution of the marginal processes $X(t)$ (and equivalently $Y(t)$) satisfies the following third-order partial differential equation
\begin{equation}\label{equazioneLeggeMarginaleRiflessione}
\frac{\partial^3p}{\partial t^3} +\frac{8\lambda(t)}{3}\frac{\partial p^2}{\partial t^2} +\frac{4}{3}\Bigl(\frac{4}{3}\lambda^2(t)+\lambda'(t) \Bigr) \frac{\partial p}{\partial t}= c^2\frac{\partial^3 p}{\partial t\partial x^2}+\frac{2c^2 \lambda(t)}{3}\frac{\partial^2 p}{\partial x^2}.
\end{equation}
Finally, if $\int_0^t\lambda(s)\dif s<\infty$, 
\begin{equation*}\label{}
P\{X(t) = ct\}=P\{X(t) = -ct\}=\frac{e^{-\int_0^t\lambda(s)\dif s}}{4},\ \ \ \ P\{X(t) = 0\} = \frac{e^{-\frac{2}{3}\int_0^t\lambda(s)\dif s}}{2}.
\end{equation*}
\end{theorem}

\begin{proof}
In order to derive the absolutely continuous component of the distribution of the projection process $\{X(t)\}_{t\ge0}$ we use probabilities (\ref{leggiMotoUnidimensionaleCongiuntoV}) with $X$ instead of $\mathcal{T}$. The system governing the functions reads
\begin{equation}\label{sistemaMarginaleRiflessione}
\begin{cases}
\frac{\partial f_1}{\partial t} = -c\frac{\partial f_1}{\partial x}+\frac{\lambda(t)}{3}(f_0+f_2-3f_1),\\
\frac{\partial f_0}{\partial t} =\frac{2}{3}\lambda(t)(f_1+f_2-f_0),\\
\frac{\partial f_2}{\partial t} = c\frac{\partial f_2}{\partial x}+\frac{\lambda(t)}{3}(f_0+f_1-3f_2).
\end{cases}
\end{equation}
Some explanation is needed for the second equation of (\ref{sistemaMarginaleRiflessione}). It follows by writing
\begin{equation}\label{equazioneEstesaF0}
f_0(x,t+\dif t) = f_0(x,t)(1-\lambda(t)\dif t)+f_0(x,t)\frac{1}{3}\lambda(t)\dif t+\bigl(f_1(x,t)+f_2(x,t)\bigr)\frac{2}{3}\lambda(t)\dif t + o(\dif t).
\end{equation}
The second term of the second member of (\ref{equazioneEstesaF0}) must be interpreted by considering that if $X(t)$ is stopped at $x$, the planar motion is moving vertically and if a Poisson event occurs, with probability $2/3$ the particle starts moving horizontally and with probability $1/3$ reflects and thus the $x$-coordinate does not change. By $f_1(x,t)2\lambda(t)\dif t /3$ we represent the probability that the particle is running rightward and the new direction is either vertical upwards or vertical downwards, the same explanation holds  for the term with $f_2$.
\\Now, the claimed result follows by proceeding as in the proof of Theorem \ref{teoremaMotoConStop}.
\end{proof}

We observe that the marginal processes $X(t)$ and $Y(t)$ behave like a one-dimensional process $\{\mathcal{T}(t)\}_{t\ge0}$ whose changes of direction are paced by $N(t)$, and that is defined as $\mathcal{T}(t) = \int_0^t V(s)\dif s$, where $\{V(t)\}_{t\ge0}$ describes the random velocity such that $V(0) = V_0$,
\[V_0 = \begin{cases}\begin{array}{r l}
+c, & w.p.\ 1/4,\\0, & w.p.\ 1/2,\\-c ,& w.p.\ 1/4,
\end{array}\end{cases}\text{and} \ \ \begin{cases}
P\{V(t+\dif t)=0\,|\,V(t) = 0,\,N(t,t+\dif t]=1\}=1/3,\\
P\{V(t+\dif t)=c\,|\,V(t) = 0,\,N(t,t+\dif t]=1\}=1/3,\\
P\{V(t+\dif t)=0\,|\,V(t) = c,\,N(t,t+\dif t]=1\}=2/3,\\
P\{V(t+\dif t)=-c\,|\,V(t) = c,\,N(t,t+\dif t]=1\}=1/3,
\end{cases}\]
and equivalently by replacing $-c$ with $c$ and vice versa.

\subsubsection*{Explicit representation of the reflecting planar motion.}
Let us assume $\Lambda(t) = \int_0^t\lambda(s)\dif s<\infty$. Here we explore the rotation (\ref{decomposizioneXYUVRiflessione}) of the reflecting orthogonal motion $\bigl(X(t),Y(t)\bigr),\ t\ge0$. We previously explained that the processes $\{U(t)\}_{t\ge0}$ and $\{V(t)\}_{t\ge0}$ appearing in (\ref{decomposizioneXYUVRiflessione}) are two dependent one-dimensional telegraph processes with rate function $2\lambda(t)/3$ and velocity $c/2$. We now describe their dependence in terms of the Poisson processes, of rate $2\lambda(t)/3$, governing the changes of velocities of $U$ and $V$, denoted by $\{N^{(U)}(t)\}_{t\ge0}$ and $\{N^{(V)}(t)\}_{t\ge0}$ respectively. For $t\ge0$, we write that
\begin{equation}\label{decomposizioneProcessiPoisson}
N^{(U)}(t) = N_U(t)+\tilde N(t),\ \ \ \ N^{(V)}(t) = N_V(t)+\tilde N(t),
\end{equation}
where $\{N_U(t)\}_{t\ge0},\ \{N_V(t)\}_{t\ge0}$ and $\{\tilde N(t)\}_{t\ge0}$ are three independent Poisson processes with rate function $\lambda(t)/3$. Clearly, the presence of $\tilde N$ in both $N^{(U)}$ and $N^{(V)}$ makes $U$ and $V$ dependent.
\\In order to explain expression (\ref{decomposizioneProcessiPoisson}), we recall that the process $\{N(t)\}_{t\ge0}$, denotes the Poisson process, with rate $\lambda(t)$, governing the changes of direction of the reflecting planar motion. Now, when $N$ indicates the occurrence of an event, the particle can:
\begin{itemize}
	\item[($i$)] switch to an orthogonal direction and therefore only one among $U=(X+Y)/2$ and $V=(X-Y)/2$ changes velocity (for instance, let the particle moving with direction $d_0$. This means that $X$ increases and therefore both $U$ and $V$ move with positive speed. If the particle deviates to the orthogonal direction $d_1$, $X$ stops and $Y$ increases, thus $U$ continues moving with positive velocity while $V$ starts moving negatively oriented. This also implies that $N_V$ perceives the event. On the other hand, if the particle deviates from $d_0$ to $d_3$, that is $Y$ decreases, $U$ starts moving backward and $V$ continues moving forward, meaning that $N_U$ makes a jump);
	\item[($ii$)] reflect to the opposite direction and therefore both $U$ and $V$ change velocity (meaning that $\tilde N$ perceives the event).
\end{itemize}
Thus, at each event counted by $N$, the process $U$ changes velocity and $V$ does not with probability $1/3$, $U$ does not change speed and $V$ does with probability $1/3$ and they both invert their velocity with probability $1/3$. This argument justifies the form (\ref{decomposizioneProcessiPoisson}) of the Poisson processes connected to $U$ and $V$. In particular, the process $\tilde N$ counts the simultaneous switches of $U$ and $V$.

Clearly, the total number of switches up to time $t\ge0,$ is given by $N(t) = N_U(t)+N_V(t) + \tilde N(t)\ a.s.$. In view of (\ref{decomposizioneXYUVRiflessione}) and (\ref{decomposizioneProcessiPoisson}) we easily obtain result (\ref{formulaGeneraleFrontieraRiflessione}). In fact, for $|\eta|<ct$
\begin{align*}
P\{&X(t)+Y(t)=ct,\, X(t)-Y(t)\in \dif \eta\} = P\Big\{U(t) = \frac{ct}{2},\,V(t)\in \frac{\dif \eta}{2}\Big\}\\
&= P\Big\{U(t) = \frac{ct}{2}\Big\}P\Big\{V(t)\in \frac{\dif \eta}{2}\,\Big|\,U(t) = \frac{ct}{2}\Big\}=P\Big\{U(t) = \frac{ct}{2}\Big\}P\Big\{V(t)\in \frac{\dif \eta}{2}\,\Big|\,\tilde N(t)=0\Big\}
\end{align*}
which coincides with (\ref{formulaGeneraleFrontieraRiflessione}) since, conditionally on $\tilde N(t)=0$, $V(t)$ is a telegraph process with changes of velocity governed by $N^{(V)}(t) = N_V(t)\sim Poisson(\Lambda(t)/3)$.

Finally, we give also the exact form of the general density of the position at time $t\ge0$. For $(x,y)\in S_{ct}\setminus \partial S_{ct}$,
\begin{align}
&P\{X(t)\in\dif x,\, Y(t)\in \dif y\} = P\Big\{U(t) \in \frac{\dif x+y}{2},\, V(t)\in \frac{x-\dif y}{2}\Big\} \nonumber\\
& = \sum_{\substack{n_U,n_V,n =0\\ n_U+n_V+n\ge1}}^\infty P\Big\{U(t) \in \frac{\dif x+y}{2},\, V(t)\in \frac{x-\dif y}{2}\,\Big|\,N_U(t)=n_U,\,N_V(t)=n_V,\,\tilde N(t)=n\Big\}\nonumber\\
&\ \ \ \times P\{N_U(t)=n_U,\,N_V(t)=n_V,\,\tilde N(t)=n\}\nonumber\\
& =  \sum_{\substack{n_U,n_V,n =0\\ n_U+n_V+n\ge1}}^\infty P\Big\{U(t) \in \frac{\dif x+y}{2}\,\Big|N^{(U)}(t)=n_U+n\Big\}P\Big\{V(t)\in \frac{x-\dif y}{2}\,\Big|N^{(V)}(t)=n_V+n\Big\}\nonumber\\
&\ \ \ \times P\{N_U(t)=n_U\}P\{N_V(t)=n_V\}P\{\tilde N(t)=n\}\nonumber\\
& = \sum_{h=0}^\infty P\Big\{U(t) \in \frac{\dif x+y}{2}\,\Big|N^{(U)}(t)=h\Big\}\sum_{k=0}^\infty P\Big\{V(t)\in \frac{x-\dif y}{2}\,\Big|N^{(V)}(t)=k\Big\}\label{leggeRiflessioneDecomposizione}\\
&\ \ \ \times \sum_{\substack{n =0\\ h+k\ge1}}^{\min\{h,k\}}P\{N_U(t)=h-n\}P\{N_V(t)=k-n\}P\{\tilde N(t)=n\}\nonumber.
\end{align}
The probabilities appearing in the last sum of (\ref{leggeRiflessioneDecomposizione}) are well-known, while the conditional distributions in the first two sums are known in the case of $\lambda(t)=\lambda>0\ \forall \ t$.
\\

We note that in the case of the standard orthogonal motion representation (\ref{decomposizioneProcessiPoisson}) holds with $\tilde N(t)=0\ \forall \ t\ a.s.$ and $N_U(t),N_V(t)\sim Poisson(\Lambda(t)/2)$, because at each event recorded by $N$ the one-dimensional process $U$ revers its speed and $V$ does not with probability $1/2$ and vice versa. The crucial fact is that $U$ and $V$ can not switch simultaneously.

\begin{remark}[Conjecture on the $L_1$-distance]
Here we conjecture a connection between the probabilities of the reflecting motion and the standard one. Let $S_u = \{(x,y)\in\mathbb{R}^2\,:\, |x|+|y|<u\}$, with $0\le u\le ct$, $p(x,y,t)$ be the absolutely continuous component of the reflecting orthogonal motion $\{\bigl(X(t),Y(t)\bigr)\}_{t\ge0}$ and $f(x,t)$ be the distribution on the $x$-diagonal of $S_{ct}$ (that coincides with the distribution on the $y$-diagonal, see Proposition \ref{propCroceRiflessione}). Now, we can write, by also assuming $Z(t)=|X(t)|+|Y(t)|,\ t\ge0$, (look at Figure \ref{L1dist})
\begin{align}
 P\{Z(t)< u\} &= P\big\{\bigl(X(t),Y(t)\bigr)\in S_{u}\big\}\nonumber\\
 & =4\int_0^u\dif x\int_0^{u-x}p(x,y,t)\dif y+2\int_{-u}^uf(x,t)\dif x.
 \end{align}
If $\Lambda(t) = \int_0^t\lambda(s)\dif s<\infty$, for $u=ct$, thanks to (\ref{riflessioneProbFrontiera}) and the probability that the motion reaches one of the vertexes, i.e. $e^{-\Lambda(t)}$, we have that
\begin{equation}\label{CongetturaDistanza2} P\{Z(t)< ct\} =P\big\{\bigl(X(t),Y(t)\bigr)\in S_{ct}\big\} = \Bigl(1-e^{-\frac{2}{3}\Lambda(t)}\Bigr)^2+e^{-\Lambda(t)}\Bigl(1-e^{-\frac{\Lambda(t)}{3}}\Bigr).
\end{equation}
The first term of the right-hand-side of (\ref{CongetturaDistanza2}) is related to the singular part of a standard orthogonal motion $\{\bigl(X_S(t),Y_S(t)\bigr)\}_{t\ge0}$ with rate function $2\lambda(t)/3$ and velocity $c$. The second part can be interpreted as
$$P\{N(t)= 0\}\int_{-ct}^{ct}P\{\mathcal{T}(t)\in \dif x\} = e^{-\Lambda(t)}\Bigl(1-e^{-\frac{\Lambda(t)}{3}}\Bigr)$$
where $\mathcal{T}$ is a one-dimensional telegraph process with parameters $(\lambda(t)/3, c/2)$, independent of the Poisson process.
Since this decomposition is true on the border we can imagine to extend it to a general square $S_u$. Therefore, we may conjecture that, for $0\le u\le ct$ and by considering the above notation,
\begin{align}
P&\{Z(t)<u\} = P\big\{\bigl(X(t),Y(t)\bigr)\in S_{u}\big\}\label{congetturaDistanzaRiflessione} \\
&=4\int_0^u\dif x\int_0^{u-x} P\{X_S(t)\in\dif x,\, Y_S(t)\in\dif y\}+P\{N(t)=0\}\int_{-u}^{u} P\{\mathcal{T}(t)\in \dif x\}.\nonumber
\end{align}
Thanks to Theorem \ref{teoremaDecomposizioneXYUV} and the known literature concerning the telegraph motion, all the distributions appearing in the second term of (\ref{congetturaDistanzaRiflessione}) are known in the case where the rate function is either $\lambda(t) = \lambda$ or $\lambda(t)= \lambda\,\text{th}(\lambda t)$, with $\lambda>0$.
\end{remark}

\subsection{Reflecting motion with Bernoulli trials}

Here, we study a reflecting orthogonal planar random motion $\{\bigl(X(t),Y(t)\bigr)\}_{t\ge0}$ which can skip the change of direction with probability $1-q\in[0,1)$. This means that when a Poisson event occurs, the particle continues with the same direction with probability $1-q$ and it switches to each of the other possible directions with probability $q/3$. We call this process \textit{$q$-reflecting orthogonal motion}, $q\in(0,1]$ (or reflecting orthogonal motion with Bernoulli trials).

\begin{theorem}
The $q$-reflecting orthogonal planar motion $\{\bigl(X_q(t),Y_q(t)\bigr)\}_{t\ge0}$, with $q\in(0,1]$ and rate function $\lambda(t)\in C^2\bigl((0,\infty),[0,\infty)\bigl)$ is equal in distribution to a reflecting orthogonal planar motion with rate function $q\lambda(t)$.
\end{theorem}

\begin{proof}
Clearly, at time $t\ge0$, the set of possible positions of the moving particle is the square $S_{ct}$, defined in (\ref{supportoQuadrato}). 

If $\Lambda(t) = \int_0^t\lambda(s)\dif s<\infty$, the singular component of the distribution is composed of the border $\partial S_{ct}$ and the diagonals of the square, $\partial D_{ct}$, defined above (see the beginning of Section \ref{Sezione3}). Let $V_{ct}=\{(0,\pm ct),(\pm ct,0)\}$, then
\begin{equation*}
P\{\bigl(X_q(t),Y_q(t)\bigr) \in V_{ct}\} = \sum_{n=0}^\infty P\{N(t)=n\}(1-q)^n = e^{-q\Lambda(t)},
\end{equation*}
because the particle reaches a vertex if it never changes direction.
\begin{align*}
P\{\bigl(X_q(t),Y_q(t)\bigr) \in \partial S_{ct}\setminus V_{ct}\} &= \sum_{n=1}^\infty P\{N(t)=n\} \sum_{k=0}^{n-1}(1-q)^k\frac{2q}{3}\Bigl(\frac{q}{3}+(1-q)\Bigr)^{n-k-1} \\
&= 2\Bigl(e^{-\frac{2q\Lambda(t)}{3}}-e^{-q\Lambda(t)}\Bigr),\nonumber
\end{align*}
where the second sum represents all the possible steps where the particle continues along the initial direction for $k$ times (with probability $(1-q)^k$), then it changes to one of the two orthogonal directions ($2q/3$) and finally it keeps moving towards the edge by choosing between this last direction and the starting one for the remaining $n-k-1$ displacements.
\\The probability of remaining on the diagonals is instead equal to
\begin{align*}
P\{\bigl(X_q(t),Y_q(t)\bigr) \in\partial D_{ct}\setminus V_{ct}\} &= \sum_{n=0}^\infty P\{N(t)=n\}\sum_{k=0}^{n-1}\binom{n}{k}(1-q)^k\Bigl(\frac{q}{3}\Bigr)^{n-k} \\
&=e^{-\frac{2q\Lambda(t)}{3}}-e^{-q\Lambda(t)},
\end{align*}
where the second sum represents all the possible sequences of directions containing only the initial one and the opposite one, which appears at least once.

By applying notation (\ref{notazioneFrontieraStandard}) for the $q$-reflecting motion, it is easy to obtain that $f_0,f_1$ satisfy the differential system (\ref{sistemiFrontieraRiflessioneG}) with $q\lambda(t)$ replacing $\lambda(t)$. Therefore, the density on the border $\partial S_{ct}^{(1)}$ satisfies system (\ref{sistemaFrontieraAltoDxRiflessione}) with $q\lambda(t)$ in place of $\lambda(t)$. This is sufficient to prove that the stated equality in distribution holds on the border of the square $S_{ct}$.

Similarly, with Proposition \ref{propCroceRiflessione} at hand, we obtain that the equality in distribution holds on the diagonals of the support by observing that $f(x,t) \dif x= P\{X_q(t)\in \dif x, Y_q(t)=0\}$ satisfies system (\ref{sistemaFrontieraAltoDxRiflessione}) with $q\lambda(t)$ replacing $\lambda(t)$.

For the absolutely continuous component we consider notation (\ref{notazione}). For $(x,y)\in S_{c(t+\dif t)}$
\begin{align*}
f_0(x,y,& t+\dif t) = f_0(x-c\dif t,y, t)\bigl(1-\lambda(t)\dif t\bigr) + f_0(x-c\dif t,y, t)\lambda(t)\dif t(1-q)\\
&+\Bigl(f_1(x,y-c\dif t,t) +f_2(x+c\dif t,y,t)+ f_3(x,y+c\dif t,t)\Bigr)\frac{q}{3}\lambda(t)\dif t +o(\dif t)
\end{align*}
and similarly for $f_1,f_2$ and $f_3$. These relationships yield system (\ref{sistemaRiflessione}) with $q\lambda(t)$ instead of $\lambda(t)$. This is sufficient to conclude the proof of the theorem.
\end{proof}

Clearly, under Kac's conditions, the $q$-reflecting orthogonal motion converge to planar Brownian motion with diffusivity $3\sigma^2/(4q)$.

\begin{remark}[Uniform orthogonal planar motion]
A particular case of the $q$-reflecting motion is the uniform orthogonal random motion, that is the process describing the position of a particle that, at each Poisson event, uniformly chooses the new direction among all the four possible directions. We obtain this motion if $q =3/4$. Note that, if $\lambda(t) = \lambda>0,\,t\ge 0$, the motion replicates at each Poisson event independently on the previous displacements and velocities. \hfill$\diamond$
\end{remark}

\begin{remark}[General random motion with Bernoulli trials]
It is interesting to observe that the behavior presented for the motions with Bernoulli trials, meaning with a positive probability to skip the change of direction, can be easily extended. In particular, a $q$-standard or $q$-reflecting motion with rate function $\lambda(t)>0,\ t\ge0$, and a time-varying probability of change of direction $q(t)\in (0,1]$, is equal in distribution to the ``basic'' motion with rate function $\lambda_q(t) = q(t)\lambda(t)$.

More generally, we can state the following.
\\\textit{Let $X_q = \{X_q(t)\}_{t\ge0}$ be a random motion in $\mathbb{R}^d$ moving with $n$ different velocities, $d,n\in \mathbb{N}$. The changes of direction of $X_q$ are paced by a Poisson process $\{N(t)\}_{t\ge0}$ with rate function $\lambda(t)\in C^n\bigl((0,\infty),[0,\infty) \bigr)$. Assume that $X_q$, at any Poisson event, skips the change of direction with probability $q(t)\in C^n\bigl((0,\infty),[0,1] \bigr)$. Then, $X_q$ is equal in distribution to a motion $\{X(t)\}_{t\ge0}$ that behaves like $X_q$, but it can not skip the switch of direction and has rate function $\lambda_q(t) = q(t)\lambda(t)$.}

The statement can be proved by observing that the point process governing the changes of direction of $X_q$ is equal in distribution to a Poisson process with rate function $\lambda_q(t)$. 
\hfill$\diamond$
\end{remark}








\footnotesize{

}

\end{document}